\documentclass[10pt,letterpaper, boxed, cm]{amsart}

\usepackage[margin=1in]{geometry}
\usepackage{amsmath,amsthm,amsfonts,amssymb,amscd}
\usepackage{lastpage}
\usepackage{fancyhdr}
\usepackage{mathrsfs}
\usepackage{xcolor}
\usepackage{graphicx}
\usepackage{listings}
\usepackage{hyperref}
\usepackage{enumitem}
\usepackage{relsize}
\usepackage{tikz-cd}
\usepackage{mdwlist}
\usepackage{multicol}
\usepackage{stmaryrd}
\usepackage{float}
\usepackage{adjustbox}
\usepackage{tikz}
\usetikzlibrary{matrix, calc, arrows}

\hypersetup{
    colorlinks=true, 
    linktoc=all,     
    linkcolor=blue,  
}

\setlist[enumerate]{label=(\alph*)}

\usetikzlibrary{graphs,decorations.pathmorphing,decorations.markings}
\tikzcdset{scale cd/.style={every label/.append style={scale=#1},
        cells={nodes={scale=#1}}}}

\newcommand{\Z}{\mathbb{Z}}

\newcommand{\Q}{\mathbb{Q}}
\newcommand{\N}{\mathbb{N}}

\newcommand{\Hom}{\operatorname{Hom}}

\newcommand{\bigO}{\mathcal{O}}

\newcommand{\tr}{\operatorname{tr}}
\newcommand{\rk}{\operatorname{rk}}

\newcommand{\inv}{^{-1}}

\newcommand{\End}{\operatorname{End}}

\newcommand{\Res}{\operatorname{Res}}
\newcommand{\Ind}{\operatorname{Ind}}
\newcommand{\Inf}{\operatorname{Inf}}
\newcommand{\Def}{\operatorname{Def}}
\newcommand{\Iso}{\operatorname{Iso}}
\newcommand{\id}{\operatorname{id}}

\newcommand{\im}{\operatorname{im}}

\newcommand{\TI}{\operatorname{ten}}

\newcommand{\Syl}{\operatorname{Syl}}

\newcommand{\triv}{\mathbf{triv}}

\newcommand{\calE}{\mathcal{E}}
\newcommand{\calF}{\mathcal{F}}

\newcommand{\calH}{\mathcal{H}}

\newcommand{\calL}{\mathcal{L}}

\newcommand{\calO}{\mathcal{O}}

\newcommand{\calS}{\mathcal{S}}
\newcommand{\calT}{\mathcal{T}}

\newcommand{\calX}{\mathcal{X}}
\newcommand{\calY}{\mathcal{Y}}

\newcommand{\catmod}{\mathbf{mod}}

\newtheorem{theorem}{Theorem}[section]
\newtheorem{lemma}[theorem]{Lemma}
\newtheorem{prop}[theorem]{Proposition}
\newtheorem{corollary}[theorem]{Corollary}

\newtheorem*{theorem*}{Theorem}

\theoremstyle{remark}
\newtheorem{remark}[theorem]{Remark}
\newtheorem{example}[theorem]{Example}

\theoremstyle{definition}
\newtheorem{definition}[theorem]{Definition}

\begin{document}
	\title{Endotrivial complexes}
    \author{Sam K. Miller}
    \address{Department of Mathematics, University of California Santa Cruz, Santa Cruz, CA 95064} 
    \email{sakmille@ucsc.edu} 
    \begin{abstract}
		Let $G$ be a finite group, $p$ a prime, and $k$ a field of characteristic $p$. We introduce the notion of an endotrivial chain complex of $p$-permutation $kG$-modules, and study the corresponding group $\mathcal{E}_k(G)$ of endotrivial complexes. Such complexes are shown to induce splendid Rickard autoequivalences of $kG$. The elements of $\mathcal{E}_k(G)$ are determined uniquely by integral invariants arising from the Brauer construction and a degree one character $G \to k^\times$. Using ideas from Bouc's theory of biset functors, we provide a canonical decomposition of $\mathcal{E}_k(G)$, and as an application, give complete descriptions of $\mathcal{E}_k(G)$ for abelian groups and $p$-groups of normal $p$-rank 1. Taking Lefschetz invariants of endotrivial complexes induces a group homomorphism  $\Lambda: \mathcal{E}_k(G) \to O(T(kG))$, where $O(T(kG))$ is the orthogonal unit group of the trivial source ring. Using recent results of Boltje and Carman, we give a Frobenius stability condition elements in the image of $\Lambda$ must satisfy.
	\end{abstract}

    \maketitle

	\section{Introduction}

	Let $G$ be a finite group, $p$ a prime, and $k$ a field of characteristic $p$. Endotrivial $kG$-modules, i.e. $kG$-modules $M$ which satisfy $M^* \otimes_k M \cong k \oplus N$ for some projective $kG$-module $N$, are objects of interest for group theorists and representation theorists. The work of Dade, Puig, and many others has elucidated much about endotrivial modules and their corresponding group $\calT(G)$. Such modules arise frequently in the theory of modular representations. For instance, endotrivial modules induce autoequivalences of the stable module category ${}_{kG}\underline\catmod$, since a $kG$-module $M$ is endotrivial if and only if $M^* \otimes_k M \cong k$ in ${}_{kG}\underline\catmod$.

	In this paper, we adapt this notion of invertibility for chain complexes. We have multiple candidate categories to choose from in order to determine what ``invertibility'' means, and in the scope of this paper, we choose the homotopy category of bounded chain complexes $K^b({}_{kG}\mathbf{mod})$. Here, contractible complexes replace the role of projective modules. Our initial definition of an endotrivial chain complex is as follows. If $C$ is a bounded complex of $kG$-modules, $C$ is endotrivial if and only if $C^* \otimes_k C \simeq k[0]$, where $k[0]$ denotes the chain complex consisting of the trivial module in degree 0 and the zero module in all other degrees, and $\simeq$ denotes homotopy equivalence, i.e. isomorphism in $K^b({}_{kG}\catmod)$. We additionally require that each module of an endotrivial chain complex is a $p$-permutation module, that is, is a permutation module when restricted to a Sylow subgroup of $G$. As a desired consequence, endotrivial complexes will induce splendid Rickard equivalences; this is Theorem \ref{srcinductionthm}.

	\begin{theorem}
		Let $C$ be an endotrivial complex of $p$-permutation $kG$-modules. Then $\Ind^{G\times G}_{\Delta G} C$, regarded as a complex of $(kG,kG)$-bimodules, induces a splendid Rickard autoequivalence of $kG$.
	\end{theorem}

	Boltje and Xu proved in \cite[Theorem 1.5]{BX07} that any splendid Rickard complex induces a $p$-permutation equivalence by taking an alternating sum of its components, i.e. taking its Lefschetz invariant (also referred to as its Euler characteristic). This induces a Lefschetz map from the set of splendid Rickard complexes to the set of $p$-permutation equivalences. We are interested in determining the image of this map, as the image may shine light on the interplay between the two types of equivalences.

	The trivial source ring $T(kG)$ is the Grothendieck group of the category of $p$-permutation modules with respect to split exact sequences, and the orthogonal unit group $O(T(kG))\leq T(kG)^\times$ is the subgroup consisting of units whose inverse is its dual. Analogously, the Lefschetz invariant of an endotrivial complex is an orthogonal unit of the trivial source ring, and we ask the same question of what the image of the Lefschetz map is. Orthogonal units induce $p$-permutation equivalences in the same manner as endotrivial complexes induce splendid Rickard complexes, so determining the image of the Lefschetz map for endotrivial complexes may be viewed as a special case of determining the image of the Lefschetz map for splendid Rickard complexes.

	A classical tool for studying $p$-permutation modules is the Brauer construction, which, given any $p$-subgroup $P$ of $G$ and $kG$-module $M$, functorially constructs a $kN_G(P)$-module $M(P)$ of which $P$ acts trivially, hence a $k[N_G(P)/P]$-module. This can be viewed as an analogue of taking $P$-fixed points of a $G$-set. For endotrivial complexes of $kG$-modules, the Brauer construction induces ``local'' endotrivial complexes of $k[N_G(P)/P]$-modules. For each $p$-subgroup of $G$, one can associate an integer corresponding to the location of non-acyclicity of the corresponding local endotrivial complex, which we call an \textit{h-mark}, analogous to the marks of a $G$-set. In fact, taking h-marks at all $p$-subgroups induces a group homomorphism, and up to a twist by a $k$-dimension one representation, completely characterizes the homotopy class of an endotrivial complex. We describe with more precision and prove these statements in Section \ref{endotrivialintroduction}.

	\begin{theorem}
		Denote by $\calE_k(G)$ the group of homotopy classes of endotrivial chain complexes of $kG$-modules, with group law given by $\otimes_k$. We have a group homomorphism \[h: \calE_k(G) \to \left(\prod_{P \in s_p(G)} \Z\right)^G\] with kernel consisting of all isomorphism classes of $k$-dimension one representations of $G$, regarded as chain complexes in degree 0.

		In particular, $\calE_k(G)$ has rank as abelian group bounded by the number of conjugacy classes of $p$-subgroups of $G$. Moreover, two endotrivial complexes are homotopy equivalent if and only if they have the same h-marks and isomorphic homology in all degrees.
	\end{theorem}

	In addition, we take inspiration from Bouc's theory of biset functors to consider the notion of the ``faithful'' subgroup of $\calE_k(G)$. We obtain a canonical decomposition of $\calE_k(G)$ into a direct product of faithful subgroups of quotient groups of $G$, which allows us to recursively determine $\calE_k(G)$ for abelian groups and groups of normal $p$-rank 1. In all such cases, we conclude that every orthogonal unit of the trivial source ring lifts to an endotrivial complex. These computations are carried out in Section \ref{section5}.

	Unfortunately, completely determining $\calE_k(G)$ for arbitrary groups does not yet appear feasible, nor does the task of determining the cokernel of $\Lambda: \calE_k(G) \to O(T(kG))$ in complete generality. In the case of $p$-groups, $O(T(kG))$ is canonically isomorphic to the unit group of the Burnside ring of $G$, $B(G)^\times$. In particular for odd $p$, it is easy to show that $\Lambda$ is surjective since in this case, $B(G)^\times = \{\pm 1\}$. The question of $p = 2$ is more difficult - in this case a basis for $B(G)^\times$ was described by Bouc in \cite[Theorem 7.4]{Bo07}, and relies on a generalized tensor induction. Unfortunately, tensor induction of chain complexes does not preserve endotriviality of complexes: we give an example in Section \ref{faithfulsection} demonstrating that even for small groups, tensor induction does not hold.

    In contrast to the previous cases, we find large classes of orthogonal units which cannot be lifted to endotrivial complexes by observing the action induced by the Frobenius endomorphism on the corresponding module categories. This theorem is the main result of Section \ref{nonsurjectivesection}. It requires some technical background, which is detailed in Remark \ref{trivsourceringdecomp}. In particular, there exists an injective homomorphism described by Boltje and Carman in \cite[Theorem A]{BC23}, $\beta_G: O(T(kG)) \to \prod_{P \in s_p(G)} R_k(N_G(P)/P)$ which associates to each orthogonal unit $u \in O(T(kG))$ a tuple $(\epsilon_P\cdot \rho_P)_{P \in s_p(G)}$ with $\epsilon_P \in \{\pm 1\}$ and $\rho_P \in \Hom(N_G(P)/P, k^\times)$.

	\begin{theorem}
		Let \[ u \in (B(S)^G)^\times \times \calL_G \leq O(T(kG)) \text{  with  } \beta_G(u) = (\epsilon_P \cdot \rho_P)_{P\in s_p(G)},\]and suppose $\rho_1$ is the trivial character on $G$. If there exists an endotrivial complex $C$ for which $\Lambda(C) = u$, then for all $p$-subgroups $P$ of $G$, $\rho_P$ admits values in $\mathbb{F}_p^\times$.
	\end{theorem}

	The paper is structured as follows. Section \ref{prelims} recalls most of the preliminary definitions and theorems needed for the paper. Section \ref{endotrivialintroduction} introduces endotrivial complexes, the group $\calE_k(G)$, and the h-mark homomorphism $h$, and establishes key structural properties of $\calE_k(G)$. Section \ref{basicpropsofendotrivs} describes the behavior of $\calE_k(G)$ with regards to functorial constructions, describes how one can partially recover the Lefschetz invariant of an endotrivial complex from the h-marks, and proves that endotrivial complexes induce splendid autoequivalences of $kG$. Section \ref{faithfulsection} introduces the faithful component of $\calE_k(G)$, gives a canonical decomposition of $\calE_k(G)$ into faithful components of quotient groups, and describes which modules can be components of indecomposable faithful complexes. Section \ref{section5} computes $\calE_k(G)$ for abelian, dihedral, generalized quaternion, and semidihedral groups. Finally, Section \ref{nonsurjectivesection} provides a necessary condition elements of $O(T(kG))$ must satisfy to be contained in the image of the Lefschetz homomorphism.

	\textbf{Notation:} Throughout, unless stated otherwise, $G$ is a finite group, $p$ a prime, and $(K, \bigO, k)$ is a $p$-modular system. That is, $\bigO$ is a complete discrete valuation ring of characteristic 0, $K$ is its field of fractions and $k$ its residue field of characteristic $p$. Note that we make no assumptions about $(K, \bigO, k)$ being large enough for $G$. All modules are finitely generated.

	We denote the set of subgroup of $G$ by $s(G)$ and the set of $p$-subgroups of $G$ by $s_p(G)$. Moreover, we write $[s(G)]$ (resp. $[s_p(G)]$) to denote a set of representatives of the conjugacy classes of $s(G)$ (resp. $s_p(G)$). If two subgroups $H_1, H_2 \leq G$ are $G$-conjugate, we write $H_1 =_G H_2$, and if $H_1$ is a subgroup of $H_2$ up to $G$-conjugacy, we write $H_1 \leq_G H_2$. If $K, H \leq G$, we write $[G/H]$ to denote a set of coset representatives of $G/H$ and write $[K\backslash G/H]$ to denote a set of double coset representatives of $K \backslash G/H$. $\Syl_p(G)$ denotes the set of Sylow $p$-subgroups of $G$.

	Let $R$ be a commutative ring. ${}_{RG}\mathbf{mod}$ denotes the category of finitely generated $RG$-modules, $Ch({}_{RG}\mathbf{mod})$ (resp. $Ch^b({}_{RG}\mathbf{mod})$) denotes the category of chain complexes (resp. bounded chain complexes) of finitely generated $RG$-modules, and $K({}_{RG}\mathbf{mod})$ (resp. $K^b({}_{RG}\mathbf{mod})$) denotes the homotopy category (resp. bounded homotopy category) of ${}_{RG}\mathbf{mod}$. Given a $RG$-module $M$, we write $M^*$ for the $R$-dual of $M$, the $RG$-module $\Hom_R(M,R)$ with left $RG$-action defined by $(g\cdot f)(m) = f(g\inv m)$. Given an $(RG,RH)$-bimodule $M$, we write $M^*$ for the $R$-dual of $M$, the $(RH,RG)$-bimodule $\Hom_R(M, R)$ with bimodule action defined by $(h\cdot f\cdot g)(m) = f(g m h)$. Finally, given two $RG$-modules $M, N$, the tensor product $M \otimes_R N$ is again an $RG$-module, with diagonal $G$-action $g\cdot (m\otimes n) = (gm \otimes gn)$.

	\section{Preliminaries}\label{prelims}

	In this section, we recall some standard facts about $p$-permutation modules, the Brauer construction, homotopy theory for chain complexes, endotrivial modules, and the Burnside ring.

	\subsection{$p$-permutation modules and the Brauer construction}

	\begin{definition}
		A $kG$-module $M$ is a \textit{$p$-permutation module} if $\Res^G_P M$ is a permutation module for all $p$-subgroups $P\in s_p(G)$.
	\end{definition}

	Note that it suffices to check $\Res^G_S M$ is a permutation module for $S \in \Syl_p(G)$. When working over a field of characteristic $p$, we have an equivalent definition of $p$-permutation modules. Given $H \leq G$, a $kG$-module $M$ is \textit{relatively $H$-projective} if there exists a $kH$-module $N$ for which $M$ is a direct summand of $\Ind^G_H N$. If $\calX$ is a set of subgroups of $G$, we say that $M$ is $\calX$-projective if each indecomposable direct summand of $M$ is $H$-projective for some $H \in \calX$. If $M$ is indecomposable, we say $M$ has \textit{vertex $P$} and \textit{source $N$} if $P \leq G$ is minimal with respect to the property that $M$ is $P$-projective, and $N$ is an indecomposable $kP$-module for which $M$ is a direct summand of $\Ind^G_P N$. It is well-known that the set of vertices of an indecomposable module form a full conjugacy class of $p$-subgroups of $G$.

	\begin{theorem}
		Let $M$ be a $kG$-module. $M$ is a $p$-permutation module if and only if $M$ is a direct summand of a permutation module. Moreover, if $M$ is indecomposable, then $M$ is a $p$-permutation module if and only if $M$ has trivial source.
	\end{theorem}
    \begin{proof}
        See \cite[Theorem 5.10.2 and Theorem 5.11.2]{L181}.
    \end{proof}

	For this reason, indecomposable $p$-permutation modules are also referred to as \textit{trivial source modules}. We write ${}_{kG}\textbf{triv}$ to denote the full subcategory of ${}_{kG}\mathbf{mod}$ with objects given by all $p$-permutation modules. It is an additive category which is symmetric monoidal and idempotent complete, but is in general not pre-abelian, since kernels and cokernels of homomorphisms of $p$-permutation modules need not be $p$-permutation. It is the idempotent completion of the category ${}_{kG}\textbf{perm}$ of permutation $kG$-modules.

	For finite $p$-groups, the trivial source modules are easy to describe.

	\begin{prop}
		Let $P$ be a finite $p$-group. Then the isomorphism classes of indecomposable $p$-permutation modules are given by $k[P/Q]$, where $Q$ runs through all subgroups of $P$.
	\end{prop}
    \begin{proof}
        See \cite[Corollary 1.11.4]{L181}.
    \end{proof}

	\begin{remark}
		Similarly, one can define $p$-permutation $\bigO G$-modules in the same way. We have a canonical functor $k\otimes_\bigO -: {}_{\bigO G}\triv \to {}_{kG}\triv$ which is essentially surjective on objects and surjective on morphisms, see \cite[Theorem 5.11.2]{L181}. In fact, this functor induces a one-to-one correspondence between the objects of ${}_{\bigO G}\triv$ and ${}_{kG}\triv$, that is, given a $p$-permutation $kG$-module $M$, there is a unique (up to isomorphism) $p$-permutation $\bigO G$-module $\widehat{M}$ such that $k \otimes_\bigO \widehat{M}\cong M$.

		In the sequel, we will work with chain complexes of $p$-permutation $kG$-modules. In general, we again have a functor $Ch({}_{\bigO G}\triv) \to Ch({}_{kG}\triv)$, however it is no longer essentially surjective, since if one attempts to lift a chain complex of $p$-permutation $kG$-modules to a chain complex of $p$-permutation $\bigO G$-modules, the resulting differentials of the resulting graded object may no longer satisfy $d^2 = 0$. For example, the 3-term chain complex of $p$-permutation $\mathbb{F}_2C_2$-modules $\mathbb{F}_2 \to \mathbb{F}_2C_2 \to \mathbb{F}_2$ with each differential the unique nonzero homomorphism between the two modules has no lift to $\bigO$.
	\end{remark}

	\begin{definition}
		The \textit{trivial source ring}, denoted by $T(kG)$, is the Grothendieck ring of ${}_{kG}\triv$ with respect to split exact sequences. Given a $kG$-module $M$, we write $[M] \in T(kG)$ to denote the class of $M$ in $T(kG)$. Every element $x \in T(kG)$ can be expressed as $x = [M] - [N]$ for some $p$-permutation $kG$-modules $M$, $N$. Addition in $T(kG)$ is induced via direct sums, that is, given $p$-permutation $kG$-modules $M,N$, we have \[[M]+[N] = [M\oplus N] \in T(kG).\] $T(kG)$ is a finitely generated free abelian group with basis given by elements $[M]$, where $M$ runs through a set of representatives of the isomorphism classes of trivial source $kG$-modules.

        Multiplication in $T(kG)$ is induced by tensor products, that is, given $p$-permutation $kG$-modules $M,N$, we have \[[M]\cdot [N] = [M \otimes_k N].\] Taking the $k$-dual of a $p$-permutation $kG$-module induces a ring automorphism $-^*: T(kG) \to T(kG)$ which is self-inverse. $O(T(kG))$ denotes the subgroup of the unit group $T(kG)^\times$ consisting of \textit{orthogonal units}, i.e. units $x \in T(kG)^\times$ for which $x^* = x\inv$.

		Denote by $R_k(G)$ the \textit{Brauer character ring} of $kG$. It is isomorphic to the Grothendieck group of ${}_{kG}\triv$ with respect to short exact sequences, and therefore we have a canonical surjection $T(kG) \to R_k(G)$. It is well-known that $R_k(G)^\times = \{\pm \chi \mid \chi \in \Hom(G,k^\times)\}$, that is, the units are given by virtual degree one Brauer characters. We refer the reader to \cite{BC23} for more details on $T(kG)$ and $R_k(G)$.
	\end{definition}

	One of the tools used to study $p$-permutation modules is the \textit{Brauer construction} (also called the \textit{Brauer quotient}).

	\begin{definition}
        \begin{enumerate}
            \item Let $R$ be either $\bigO$ or $k$, and let $M$ be a $RG$-module. Given any subgroups $Q \leq P \leq G$, the \textit{trace map} $\tr^P_Q: M^Q \to M^P$ is given by \[\tr^P_Q: m \mapsto \sum_{x \in [P/Q]} x\cdot m.\]

            \item Given any subgroup $P \leq G$, the \textit{Brauer construction} at $P$, $-(P): {}_{ \bigO G}\mathbf{mod} \to {}_{k[N_G(P)]}\mathbf{mod}$ is defined by \[M(P) := M^P/\sum_{Q < P} \text{tr}^P_Q(M^Q) + J(\bigO) M^P,\] where $J(\bigO)$ denotes the Jacobson radical of $\bigO$. Given any $\bigO G$-module $M$, $P \leq \ker M(P)$, so the Brauer construction may also be considered an additive functor $(-)(P): {}_{\bigO G}\mathbf{mod} \to {}_{k[N_G(P)/P]}\mathbf{mod}$.

            \item We also have a functor $-(P): {}_{kG}\mathbf{mod} \to {}_{k[N_G(P)/P]}\mathbf{mod}$ defined by \[M(P) := M^P/\sum_{Q < P} \text{tr}^P_Q(M^Q),\] which we also refer to as the \textit{Brauer construction}. We have a commutative diagram as follows.

            \begin{figure}[H]
                \centering
                \begin{tikzcd}
                    {}_{\calO G}\textbf{mod} \ar[d, "k \otimes_\bigO -"'] \ar[dr, "-(P)"] \\
                    {}_{kG}\textbf{mod} \ar[r, "-(P)"]& {}_{k[N_G(P)/P]}\textbf{mod}
                \end{tikzcd}
            \end{figure}

            In this paper, we primarily use the Brauer construction  $-(P): {}_{kG}\mathbf{mod} \to {}_{k[N_G(P)/P]}\mathbf{mod}$.
        \end{enumerate}

	\end{definition}

	If $P$ is not a $p$-subgroup of $G$, for all $kG$-modules $M$,  $M(P) = 0$, thus we restrict our attention to taking Brauer quotients at $p$-subgroups of $G$. The Brauer construction restricts to a functor $(-)(P): {}_{kG}\triv \to {}_{k[N_G(P)/P]}\triv$. It extends to a functor of chain complexes, $Ch({}_{kG}\mathbf{mod}) \to Ch({}_{k[N_G(P)/P]}\mathbf{mod})$ and similarly for ${}_{kG}\triv$. This functor is in general not exact.

	\begin{prop}\label{brauercommuteswithdualsandconj}
		\begin{enumerate}
			\item Let $M$ be a $kG$-module and $P\in s_p(G)$. $M(P)^* \cong (M^*)(P)$ is a natural isomorphism of $k[N_G(P)/P]$-modules.
			\item Given any $kG$-module $M$, $P \in s_p(G)$ with $P \leq H \leq G$, $\Res^{N_G(P)/P}_{N_H(P)/P} \big(M(P)\big) = \left(\Res_H^G M\right)(P)$.
			\item Let $M$ be a $kG$-module, $P \in s_P(G)$, and $x \in G$. ${}^x(M(P)) \cong M({}^xP)$ is a natural isomorphism of $k[{}^x(N_G(P)/P)]$-modules.
		\end{enumerate}
	\end{prop}
	\begin{proof}
	    All of these follow directly from the definition of the Brauer construction and are easily verified.
	\end{proof}

	The Brauer construction behaves especially well with $p$-permutation modules. The following are well-known properties.

	\begin{prop} \label{brauerppermproperties}
        Let $M,N$ be (finitely generated) $p$-permutation $kG$-modules and let $P \in s_p(G)$.
		\begin{enumerate}
			\item $(M\otimes_k N)(P) \cong M(P) \otimes_k N(P)$ as $k[N_G(P)/P]$-modules. This isomorphism is natural in $M$ and $N$.
			\item If $M \cong kX$ is a permutation module, then $M(P)$ is a permutation $k[N_G(P)/P]$-module with basis the image of the $P$ fixed points $X^P$ of $X$ under the quotient map $M^P \to M(P)$. In particular, $k[G/P](P) \cong k[N_G(P)/P]$ naturally as $k[N_G(P)/P]$-modules.
			\item If $M$ is a trivial source module, then $M(P) \neq 0$ if and only if $P$ is contained in a vertex of $M$.
			\item Let $Q \trianglelefteq P$. Then $M(P) \cong \left(M(Q)\right)(P)$.
		\end{enumerate}
	\end{prop}
    \begin{proof}
        (a) is \cite[Proposition 5.8.10]{L181}. (b) follows from \cite[Proposition 5.8.1]{L181}. (c) is \cite[Proposition 5.10.3]{L181}. (d) is \cite[Proposition 5.8.5]{L181}.
    \end{proof}

	The following theorems will be crucial in the sequel.

	\begin{theorem}\label{permsplit}
		Let $M, N$ be $p$-permutation $kG$-modules, and let $f: M \to N$ be an injective (resp. surjective) homomorphism. Then $f$ is split injective (resp. surjective) if and only if $f(P)$ is injective (resp. surjective) for all $p$-subgroups $P\in s_p(G)$.
	\end{theorem}
    \begin{proof}
        \cite[Proposition 5.8.11]{L181} implies the result for the surjective case. Now, since $k$ is a field, $f: M\to N$ is split injective if and only if $f^*: N^* \to M^*$ is split surjective if and only if $f^*(P)$ is surjective for all $P \in s_p(G)$ if and only if $f(P)^*$ is surjective for all $P \in s_p(G)$ (due to Proposition \ref{brauercommuteswithdualsandconj} (a)) if and only if $f(P)$ is injective for all $P \in s_p(G)$.
    \end{proof}

	\begin{theorem}\label{ppermvertex}
		Let $U$ be an indecomposable trivial source $kG$-module with vertex $P$, and let $M$ be a $p$-permutation $kG$-module. Let $f: U \to M$ be a $kG$-module homomorphism. The following are equivalent:
		\begin{itemize}
			\item $f: U \to M$ is split injective.
			\item $f(P): U(P) \to M(P)$ is injective.
		\end{itemize}
		Dually, let $g: M\to U$ be a $kG$-module homomorphism. The following are equivalent:
		\begin{itemize}
			\item $g: M \to U$ is split surjective.
			\item $g(P): M(P)\to U(P)$ is surjective.
		\end{itemize}
	\end{theorem}
    \begin{proof}
        The injective case is precisely \cite[Propositionn 5.10.7]{L181}, and since $k$ is a field, the dual statement follows by a similar argument as in the previous proof.
    \end{proof}

	For a more detailed treatment of $p$-permutation modules, the trivial source ring, and the Brauer quotient, we direct the reader to \cite{CL23}, \cite{L181}, or \cite{BC23}.

	\subsection{Chain complexes}

	Next, we recall the definitions of the tensor product and internal hom of chain complexes, following \cite{W94}. We refer the reader to \cite{W94} or \cite[Sections 1.17 and 1.18]{L181} for a more comprehensive overview of homological algebra. For this section, let $R$ be a commutative ring.

	\begin{definition}
		Let $R$ be a commutative ring and $C, D$ two bounded chain complexes of $RG$-modules with differentials denoted by $c_i, d_j$ respectively.
		\begin{enumerate}
			\item The \textit{tensor product} is a chain complex of $RG$-modules, denoted $C\otimes_R D$, and is defined as follows.
			\begin{enumerate}
				\item $(C\otimes_R D)_n = \bigoplus_{i+j = n} C_i \otimes_R D_j$.
				\item $d^C_{i,j}: C_i \otimes_R D_j \to C_{i-1} \otimes_R D_j$ is the homomorphism $c_i \otimes \id$.
				\item $d^D_{i,j}: C_i \otimes_R D_j \to C_{i} \otimes_R D_{j-1}$ is the homomorphism $(-1)^i\id \otimes d_j$.
			\end{enumerate}

			\item The \textit{internal hom} is a chain complex of $RG$-modules, denoted $\Hom_{R}(C,D)$, and is defined as follows.
			\begin{enumerate}
				\item $\Hom_R(C,D)_n = \bigoplus_{j-i = n} \Hom_R(C_i, D_j)$.
				\item $d^C_{i,j}: \Hom_R(C_i, D_j) \to \Hom_R(C_{i+1}, D_j)$ is the homomorphism $(-1)^{1 + j - i}(c_{i+1})^*$, where $(-)^*$ denotes precomposition.
				\item $d^D_{i, j}: \Hom_R(C_i, D_j) \to \Hom_R(C_i, D_{j-1})$ is the homomorphism $(d_j)_*$, where $(-)_*$ denotes postcomposition.
			\end{enumerate}
			Given a $RG$-module $M$ and $i \in \Z$, we write $M[i]$ for the chain complex with $M$ in degree $i$ and zero modules in all other degrees. The \textit{dual chain complex} is $C^* = \Hom_R(C, R[0])$ with the above sign conventions.
		\end{enumerate}
	\end{definition}

	Two chain complexes $C,D$ of $kG$-modules are homotopy equivalent, denoted $C \simeq D$, if and only if they are isomorphic in the homotopy category $K^b({}_{RG}\mathbf{mod})$. If the Krull-Schmidt theorem holds, equivalently, there exist contractible complexes $C', D'$ of $kG$-modules such that $C \oplus C' \cong D \oplus D'$.

	Recall that a $RG$-module homomorphism $f: M \to N$ is \textit{split} if and only if there exists a homomorphism $s: N \to M$ such that $f = f \circ s \circ f$. In this case, we have $M = \ker f \oplus \im s\circ f$ and $N = \im f \oplus \ker f\circ s$. A chain complex is \textit{split} if and only if all of its differentials are split. The following alternative characterization is well-known.

	\begin{prop}{\cite[Propopsition 1.18.15]{L181}}
		A chain complex $C \in Ch({}_{RG}\mathbf{mod})$ is split if and only if $C\simeq H(C)$, with $H(C)$ viewed as a chain complex via $(H(C))_i := H_i(C)$ and with zero maps as differentials. In particular, a chain complex is contractible if and only if it is split acyclic.
	\end{prop}

	Given a bounded chain complex $C$, if $C_i \neq 0$, $C_{i'} = 0$ for all $i' > i$, $C_j \neq 0$, and $C_{j'} = 0$ for all $j' < j$, the \textit{length} of $C$ is $(i - j) + 1$. For example, the contractible chain complex of $kG$-modules \[\dots \to 0 \to k \xrightarrow{\id} k \to 0\to \cdots\] has length 2.

    The following characterization of bounded contractible chain complexes will be used throughout the paper.

    \begin{prop}
        A bounded chain complex $C \in K^b({}_{RG}\mathbf{mod})$ is contractible, i.e. $C \simeq 0$, if and only if $C$ is isomorphic to a finite direct sum of chain complexes of the form $\dots \to 0 \to M \xrightarrow{\cong} M \to 0\to \dots$.
    \end{prop}
    \begin{proof}
        The reverse implication is clear. Suppose $C$ is contractible; equivalently it is split acyclic. Let $i \in \Z$ be the maximum integer for which $C_i \neq 0$. Since $C$ is split acyclic, $d_i: C_i \to C_{i-1}$ is split injective, so we have an isomorphism $C \cong (C_i \xrightarrow{\cong} C_i) \oplus C'$, where $(C_i \xrightarrow{\cong} C_i)$ is nonzero in degrees $i$ and $i-1$, and $C'$ is a chain complex with length one less than $C$. The result now follows by an inductive argument on the length of $C$.
    \end{proof}

	We state the version of the K\"unneth formula which will be critical for most of the techniques used in this paper. For more general versions, see \cite[Theorem 3.6.3 and Theorem 5.6.4]{W94} or \cite[Theorem 2.21.7]{L181}

	\begin{theorem}{(K\"unneth theorem)} \label{kunneth}
		Let $C$ and $D$ be complexes of $kG$-modules. We have a natural isomorphism of $kG$-modules \[\bigoplus_{p+q = n} H_p(C) \otimes_k H_q(D) \cong H_n(C\otimes_k D)\] for all $n \in \Z$.
	\end{theorem}

	\begin{proof}
		Let $i, j, n$ be integers for which $i+j = n$. We denote the differentials of $C,D,$ and $C\otimes_k D$ by $c, d,$ and $e$ respectively. First note if $x \in \ker c_i$ and $y \in \ker d_j$, then $x\otimes y \in \ker e$. Moreover, if $x \in \im c_{i+1}$ or $y \in \im d_{j+1}$, then $x \otimes y \in \im e_{n + 1}$. Therefore, the map sending a pair \[(x + \im \delta_{i+1}, y + \im \epsilon_{j+1}) \mapsto x\otimes y\] induces a well-defined $kG$-module homomorphism \[H_i(C) \otimes_k H_j(D) \to H_n(C\otimes_k D).\] We show that the direct sum of these maps taken over all pairs $(i,j)$ satisfying $i+j = n$ yields an isomorphism. Note that the map as constructed is natural in $C$ and $D$.

        To show that this map is an isomorphism, it suffices to show that it is an isomorphism of $k$-modules. $C$ and $D$ are split as complexes of $k$-modules, therefore we have $D \simeq H(D)$, so $C \otimes_k D \simeq C \otimes_k H(D)$. Now, we have an isomorphism of chain complexes of $k$-vector spaces \[C \otimes_k H(D) \cong \bigoplus_{j \in \Z} C \otimes_k H_j(D),\] where $H_j(D)$ is regarded as a chain complex concentrated in degree $j$. It follows that \[H_n(C \otimes_k H_j(D)) \cong H_{n-j}(C) \otimes_k H_j(D),\] as $k$-vector spaces, and taking the direct sum over all $j \in \Z$ implies the result.
	\end{proof}

	\subsection{Endotrivial modules}

	Endotrivial complexes are a chain complex-theoretic analogue of endotrivial modules. We present this section both for later use and for additional motivation.

	\begin{definition}
		\begin{enumerate}
            \item If $G$ is a $p$-group, a $kG$-module $M$ is an \textit{endopermutation module} if $M^* \otimes_k M$ is a permutation $kG$-module. More generally, if $G$ is a finite group, a $kG$-module $M$ is an \textit{endo-$p$-permutation module} if $M^* \otimes_k M$ is a $p$-permutation $kG$-module (see \cite{Ur06}).
			\item A $kG$-module $M$ is an \textit{endotrivial module} if $M^* \otimes_k M \cong k \oplus P$, where $P$ is a projective $kG$-module. In other words, $M^* \otimes_k M \cong k$ in the stable module category ${}_{kG}\underline{\mathbf{mod}}$.
			\item We define the group of endotrivial $kG$-modules $\mathcal{T}(G)$ as follows. The elements of $\calT(G)$ are classes of endotrivial modules which are isomorphic in the stable module category ${}_{kG}\underline\catmod$. Here, $M_1$ and $M_2$ are isomorphic in ${}_{kG}\underline\catmod$ if and only if $M_1 \oplus P_1 \cong M_2 \oplus P_2$ for some projective $kG$-modules $P_1, P_2$. Write $[M] \in \calT(G)$ to denote the isomorphism class of $M$. Addition in $\calT(G)$ is induced by $\otimes_k$, that is, given two endotrivial modules $M_1, M_2$, \[[M_1] + [M_2] := [M_1 \otimes_k M_2].\]
		\end{enumerate}
	\end{definition}

    Standard notation for the group of endotrivial modules is $T(G)$, however since we also use $T$ to denote the trivial source ring, we switch to $\calT(G)$ to distinguish between the two groups.

	\begin{example}
		Most known examples of endotrivial modules come from syzygies. Proofs of the following two examples are fairly elementary.
		\begin{enumerate}
            \item Let $\Omega(M)$ be the kernel of the projective cover $P \to M$, and define $\Omega^i(M) = \Omega(\Omega^{i-1}(M))$ for $i \in \N$, with $\Omega^1(M) = \Omega(M)$. Similarly, let $\Omega\inv(M)$ be the cokernel of the injective hull $M \to I$, and define $\Omega^{-i}(M) := \Omega\inv(\Omega^{-(i-1)}(M))$ for $i \in \N$. If $M$ is endotrivial, $\Omega^i(M)$ is endotrivial for all $i \in \Z$. See \cite[Example 2.1]{Ma18} for more details.

			\item Let $X$ be a $G$-set. The \textit{relative syzygy} $\Delta(X)$ is the kernel of the augmentation homomorphism $kX \to k$. If $G$ is a $p$-group, $\Delta(X)$ is an endopermutation module, see \cite[Theorem 1]{Al01}. Additionally, $\Delta(X)$ is an endotrivial module in some rare cases. For example, when $G$ is a semidihedral 2-group and $X = G/H$, where $H$ is a noncentral subgroup of order 2, $\Delta(G/H)$ is endotrivial, and the class of $\Omega(\Delta(G/H))$ in $\calT(G)$ is torsion with order 2 (see \cite[Theorem 7.1]{CaTh00}).
		\end{enumerate}
	\end{example}

	\begin{remark}
		Puig showed in \cite[Corollary 2.4]{Pu90} that when $G$ is a $p$-group, $\calT(G)$ is finitely generated as an abelian group. This theorem was later extended to all finite groups, see \cite[Corollary 2.5]{CMN06} for instance. Therefore $\calT(G)$ has a decomposition into its torsion subgroup and a (not necessarily unique) free subgroup. One goal of modular representation theorists is to completely determine the structure of $\calT(G)$ for all finite groups, and to determine explicit constructions for generators, an analogy that we will adopt in the sequel. In general, this is an open problem, but is known for many cases, including all $p$-groups, which was completed by Carlson and Th\'evenaz over multiple papers, see \cite{CaTh00} and \cite{CT04}. In general, the difficulty in understanding the isomorphism class of $\calT(G)$ is determining its torsion subgroup, as the torsionfree rank can be deduced from the structure of $G$.

		In general, torsion in $\calT(G)$ is rare for $p$-groups; in fact, torsion occurs only for cyclic groups of order at least 3, quaternion groups, and semidihedral groups. The group of endotrivial modules for a finite $p$-group is generated by relative syzygies, except in the case $p = 2$ and $G$ is (generalized) quaternion. We refer the reader to \cite{Ma18} for more detailed exposition on endotrivial and endopermutation modules.
	\end{remark}

	\subsection{The Burnside ring}

	We next discuss the Burnside ring $B(G)$ of a finite group, which partially governs the decomposition of $O(T(kG))$. If $P$ is a $p$-group, since the indecomposable $p$-permutation $kP$-modules are $k$-linearized transitive permutation modules, the trivial source ring of $kP$ is isomorphic to the Burnside ring. We follow \cite{Bou10} for this subsection.

	\begin{definition}
		The \textit{Burnside ring} of a finite group $G$, denoted $B(G)$, is the Grothendieck group of the additive category of finite $G$-sets, ${}_G\textbf{set}$. If $X$ is a $G$-set, we write $[X]$ to denote the image of $X$ in $B(G)$. $B(G)$ is a ring with product induced by the direct product. We define the \textit{mark homomorphism} as follows, following:
		\begin{align*}
		\mathfrak{m}: B(G) &\to \left(\prod_{H\leq G} \Z \right)^G\\
		X &\mapsto (|X^H|)_{H\leq G}
		\end{align*}
		Here, the $G$-action is given by $G$-conjugation on the poset of subgroups of $G$. It follows that the image of $\mathfrak{m}$ is $G$-stable from the property that if $K, H \leq G$, then $G/K\cong G/H$ as $G$-sets if and only $H =_G K$.
	\end{definition}

	\begin{remark}\label{burnsideringmarkhominverse}
		Note that we have a ring homomorphism $B(G) \to T(kG)$ induced by $k$-linearization, $[X] - [Y] \mapsto [kX] - [kY]$. It is in general neither injective nor surjective, but if $G$ is a $p$-group it is an isomorphism.

		$\mathfrak{m}$ is a full-rank injective ring homomorphism, so after extending scalars to $\Q$, we obtain an isomorphism
		\[\Q\otimes \mathfrak{m}: \Q \otimes_\Z B(G) \xrightarrow[]{\sim} \left(\prod_{H\leq G} \Q\right)^G\] We set $\Q B(G) := \Q \otimes_\Z B(G)$ and $\Q \mathfrak{m} := \Q \otimes \mathfrak{m}$.

		We describe the inverse $\Q \mathfrak{m}\inv$. Define the primitive idempotent $\delta_{[H]} \in \left(\prod_{H\leq G}\Q \right)^G$ by $(\delta_{[H]})_K = 1$ if $H =_G K$ and $(\delta_{[H]})_K = 0$ otherwise. Then, $\delta_{[H]}$ has preimage
		\[e^G_H =\frac{1}{|N_G(H)|}\sum_{K\leq H} |K|\mu(K, H) [G/K] \in \Q B(G), \] where $\mu$ is the M\"{o}bius function associated to the poset of subgroups of $G$. $\Q \mathfrak{m}\inv: \left(\prod_{H\leq G}\Q\right)^G \to \Q B(G)$ is as follows:

		\begin{align*}
		\Q \mathfrak{m}\inv: \left(\prod_{H\leq G}\Q\right)^G &\to \Q B(G)\\
		(a_H)_{H\leq G} &\mapsto \sum_{H \in [s(G)]} a_H e^G_H
		\end{align*}

	\end{remark}

	\section{Endotrivial complexes and h-marks} \label{endotrivialintroduction}

	In this section, we introduce the notion of an endotrivial chain complex and define the group $\calE_k(G)$ of endotrivial $kG$-complexes. We will find via the Brauer construction that elements of this group can be described via integral constants, similar to how elements of the Burnside ring can be described via their marks.

	\begin{definition}
		Let $C \in Ch^b({}_{kG}\triv)$. Say $C$ is an \textit{endotrivial complex} if $\End_k(C) \cong C^* \otimes_k C \simeq k[0]$.
	\end{definition}

	Of course, we can define endotriviality for ${}_{kG}\mathbf{mod}$, but the scope of this paper is limited to bounded chain complexes of $p$-permutation $kG$-modules. ``Endotrivial complexes'' implies bounded endotrivial complexes of $p$-permutation $kG$-modules.s

	\begin{definition}
		\begin{enumerate}
			\item We define the group $\calE_k(G)$ to be the set of all homotopy classes of endotrivial chain complexes of $p$-permutation $kG$-modules. Given an endotrivial complex $C$, we write $[C]\in \calE_k(G)$ to denote the corresponding homotopy class in $\calE_k(G)$. $\calE_k(G)$ forms an abelian group with group addition induced by $\otimes_k$, that is, given two endotrivial complexes $C_1, C_2$, \[[C_1] + [C_2] := [C_1 \otimes_k C_2].\]

			\item We call the trivial module $k$, regarded as a chain complex concentrated in degree 0, the \textit{trivial endotrivial complex}. $[k]$ is the identity element of $\calE_k(G)$.
		\end{enumerate}
	\end{definition}

	We may abusively refer to elements of $\calE_k(G)$ as chain complexes, rather than homotopy classes, when it is permissible to do so. We write $C \in \calE_k(G)$ to denote that $C$ is an endotrivial complex of $kG$-modules. One has to take care when defining properties of elements in $\calE_k(G)$ via representatives of homotopy classes, to make sure the properties are homotopy invariant.

	First, note that the Brauer construction, restriction, and inflation all preserve endotriviality.

	\begin{prop}\label{brauerpreservesendotriviality}
        If $G$ and $H$ are finite groups and $\calF: Ch^b({}_{kG}\triv) \to Ch^b({}_{kH}\triv)$ is an additive functor such that for all $C_1,C_2 \in Ch^b({}_{kG}\triv)$, $\mathcal{F}(C_1)\otimes_k \mathcal{F}(C_2) \cong \mathcal{F}(C_1\otimes_k C_2)$, $\mathcal{F}(C_1^*) \cong\mathcal{F}(C_1)^*$, and $\mathcal{F}(k) \cong k$, then $\mathcal{F}$ induces a well-defined group homomorphism $\calE_k(G) \to \calE_k(H)$.

		In particular, if $C \in \calE_k(G)$,
        \begin{itemize}
            \item For all $p$-subgroups $P \in s_p(G)$, $C(P) \in \calE_k(N_G(P)/P)$, and $-(P)$ induces a well-defined group homomorphism $\calE_k(G) \to \calE_k(N_G(P)/P)$.
            \item For all subgroups $H \leq G$, $\Res^G_H C \in \calE_k(H)$, and $\Res^G_H$ induces a well-defined group homomorphism $\calE_k(G) \to \calE_k(H)$.
            \item If $G$ is a quotient of $\tilde{G}$, $\Inf^{\tilde{G}}_G C \in \calE_k(\tilde{G})$, and $\Inf^{\tilde{G}}_G$ induces a well-defined group homomorphism $\calE_k(G) \to \calE_k(\tilde{G}). $
        \end{itemize}

	\end{prop}
	\begin{proof}
        Since $C \otimes_k C^* \simeq k$, we have the following sequence of homotopy equivalences of $kH$-complexes: \[k = \calF(k) \simeq \calF(C\otimes_k C^*) \cong \calF(C)\otimes_k \calF(C^*) \cong \calF(C) \otimes_k \calF(C)^*.\] Thus $\calF(C)$ is endotrivial. $\calF$ preserves homotopy equivalences, so the map $\calF: \calE_k(G) \to \calE_k(H)$ is well-defined, and it is a group homomorphism since $\calF$ commutes with tensor products by assumption. This holds for the Brauer construction, restriction and inflation, as they all satisfy the assumed properties.
	\end{proof}

	The following theorem is an equivalent formulation of endotriviality for $p$-permutation complexes. This was communicated to the author by Robert Boltje.

	\begin{theorem}\label{endotrivdef2}
		Let $C \in Ch^b({}_{kG}\textbf{triv})$. Then $C$ is endotrivial if and only if for all $p$-subgroups $P \in s_p(G)$, $C(P)$ has nonzero homology concentrated exactly in one degree, with the nontrivial homology having $k$-dimension one.
	\end{theorem}

	\begin{proof}
		First, suppose $C$ is endotrivial, then $C(P)$ is endotrivial as well. Since $H_i(C)^* \cong H_{-i}(C^*)$ for all $i \in \Z$, an easy argument using the K{\"u}nneth formula shows that $C(P)$ has nonzero homology concentrated in one degree, with the nonzero homology having $k$-dimension one.

		Conversely, suppose for all $p$-subgroups $P \in s_p(G)$, $C(P)$ has homology concentrated in one degree, with the homology having $k$-dimension one. We show $C\otimes_k C^* \simeq k$. By the K{\"u}nneth formula, we have that $C \otimes_k C^* \cong D$, where $D$ is a chain complex satisfying $ H_0(D) \cong k$. Label the differentials of $D$ by $d_n: D_n \to D_{n-1}$. If $C$ has length $n$, then $D$ has length $2n - 1$, with $n-1$ and $-(n-1)$ the greatest and least indices respectively for which $D_i \neq 0$. Moreover, $D(P)$ has nonzero homology concentrated in degree zero for any $p$-subgroup $P\in s_p(G)$. Indeed, if $i \in \Z$ is the unique integer for which $H_i(C(P))\neq 0$, then since $H_i(C(P))^* \cong H_{-i}(C(P)^*) \neq 0$
        \begin{align*}
            H_0((C\otimes_k C^*)(P)) &\cong H_0(C(P) \otimes_k C^*(P))\\
            &\cong H_0(C(P) \otimes_k C(P)^*) \\
            &\cong H_i(C(P)) \otimes_k H_{-i}(C(P)^*) \\
            &\cong H_i(C(P)) \otimes_k H_i(C(P))^* \cong  k
        \end{align*}

        It follows similarly by the K\"unneth theorem that $H_i((C\otimes_k C^*)(P)) = 0$ for $i \neq 0.$

		$d_{n-1}(P)$ and $d_{-n + 2}(P)$ are injective and surjective respectively over all $p$-subgroups $P\in s_p(G)$, and therefore by Theorem \ref{permsplit} are split injective and surjective respectively. Thus, we have an isomorphism \[D \cong D' \oplus (D_{n-1} \xrightarrow{\cong} D_{n-1}) \oplus (D_{-n + 2} \xrightarrow{\cong} D_{-n + 2}),\] where $D'$ is a chain complex of length $2n - 3$ for which $D'(P)$ has nonzero homology only in degree zero for all $P \in s_p(G)$. Here, $D_{n-1} \xrightarrow{\cong} D_{n-1}$ is concentrated in degrees $n-1$ and $n-2$, and $D_{-n + 2} \xrightarrow{\cong} D_{-n + 2}$ is concentrated in degrees $-n+2$ and $-n+1$. These two complexes are contractible, so $D \simeq D'$. An induction argument on the length of $C$ yields the homotopy equivalence $D \simeq k$, as desired.
	\end{proof}

	If $C,D \in Ch^b({}_{kG}\textbf{triv})$ satisfy $C \simeq D$, then $C(P) \simeq D(P)$ for all $p$-subgroups $P \leq G$. Moreover, since homology is preserved under homotopy equivalence, we may speak of ``the homology'' of a homotopy equivalence class of chain complexes.

	\begin{definition} \label{hmarkhom}
		For a class of endotrivial complexes $[C]$ and $P \in s_p(G)$, denote by $h_C(P)$ the degree $i\in \Z$ of a representative $C$ in which $H_i(C(P)) \neq 0$, and set $\calH_C(P):= H_{h_C(P)}(C(P))$. This is well-defined since the Brauer construction preserves homotopy equivalence. We will call $h_C(P)$ the \textit{h-mark of $C$ at $P$}, and refer to $\calH_C(P)$ as \textit{the homology of $C$ at $P$.} We call the function $h_C$ \textit{the h-marks of $C$.}

        Since the $k$-dimension of $\calH_C(P)$ is one, we may also consider $\calH_C(P)$ a group homomorphism $N_G(P)/P \to k^\times$, by identifying a $k$-dimension one $k[N_G(P)/P]$-module $k_\omega$ with its corresponding group homomorphism $\omega: N_G(P)/P\to k^\times$. We have a map:

		\begin{align*}
		\Xi: \calE_k(G) &\to \prod_{P \in s_p(G)} \Z \times \Hom(N_G(P)/P, k^\times)\\
		[C] &\mapsto \big(h_C(P), \calH_C(P)\big)_{P \in s_p(G)}
		\end{align*}
		It is straightforward to verify that this map is a well-defined group homomorphism via the K{\"u}nneth formula and commutativity of the Brauer construction with tensor products for $p$-permutation modules.

        For any endotrivial complex $C$, $h_C$ is a class function by Proposition \ref{brauercommuteswithdualsandconj} (c). The assignment $[C] \mapsto h_C$ is well-defined group homomorphism:

		\begin{align*}
		h: \calE_k(G) &\to \left(\prod_{P \in s_p(G)} \Z \right)^G\\
		[C] &\mapsto \big(h_C(P)\big)_{P \in s_p(G)}
		\end{align*}

        Here, the $G$-action is induced by conjugation on $s_p(G)$. Note the connection to the mark homomorphism on the Burnside ring, which is also $G$-stable under the same action. Equivalently, we may view $h$ as a group homomorphism,
        \begin{align*}
            h: \calE_k(G) &\to C(G,p)\\
            [C] &\mapsto h_C
        \end{align*}
        where $C(G,p)$ denotes the additive group of $\Z$-valued class functions on $p$-subgroups of $G$.
	\end{definition}

	\begin{theorem}\label{kernelthm}
		$\ker h = \{k_{\omega}[0] \mid \omega \in \Hom(G, k^\times)\}$. In other words, up to homotopy, the only endotrivial complexes $C$ such that $C(P)$ has homology in degree 0 for all $p$-subgroups $P \in s_p(G)$ are the $k$-dimension one representations of $kG$, regarded as chain complexes concentrated in degree 0.
	\end{theorem}
	\begin{proof}
		Suppose for contradiction there is another class of endotrivial complexes $[C]$ for which $h_C$ is the zero function. Choose $C$ to be a representative such that $C$ is minimal with respect to length. $C$ must have at least length two by assumption. Let $i \in \Z$ be the greatest integer such that $C_i \neq 0$ and let $j \in \Z$ be the least integer such that $C_j \neq 0$ (so $C$ has length $i - j + 1$). Since $h_C = 0$, $d_0$ is a nonzero homomorphism, so $C_0 \neq 0$. Since the length of $C$ is at least 2, either $i > 0$ or $j < 0$. We assume  $i > 0$, the other case follows dually.

		Since $H_i(C) = 0$ (as $h_C = 0$), $d_i: C_i \to C_{i-1}$ is injective. Now, for all $p$-subgroups $P\in s_p(G)$, the largest integer $i'$ for which $C_{i'}(P) \neq 0$ satisfies $i' \leq i$. Moreover since $h_C(P) = 0$, $H_i(C(P)) = 0$ as well, so $d_i(P)$ is injective as well (possibly the zero map). Since $P$ is an arbitrary $p$-subgroup of $G$, Theorem \ref{permsplit} implies $d_i$ is split injective. It follows that $C$ is homotopy equivalent to a chain complex $C'$ of length $i+j-2$. This contradicts minimality of the length of $C$, so we are done.
	\end{proof}

	\begin{corollary}
		$\Xi$ is injective. In particular, given integers $(x_P)_{P\in s_p(G)}$ and a linear character $\rho \in \Hom(G,k^\times)$, there is at most 1 element $(a_P, \rho_P)_{P\in s_p(G)} \in \im \Xi$ for which  $\rho_1 = \rho$ and $a_P = x_P$ for all $P\in s_p(G)$.

		In particular, $\calE_k(G)$ is finitely generated, with $\Z$-rank bounded by the number of conjugacy classes of $p$-subgroups of $G$ and torsion subgroup isomorphic to $\Hom(G,k^\times)$.
	\end{corollary}
	\begin{proof}
	    This follows immediately from the previous theorem.
	\end{proof}

	\begin{remark}
		We will see in the sequel that $h$ is rarely a full-rank homomorphism.

		We can check if two endotrivial complexes are homotopy equivalent up to a twist by a $k$-dimension one representation by checking the degrees of their homology at every pair of complexes induced by the Brauer construction. In particular, if $G$ is a $p$-group, then $h$ is injective since the only $k$-dimension one $kG$-module is $k$ itself.

		In fact, $h$ yields a split exact sequence:
		\[0 \to \Hom(G, k^\times) \to \calE_k(G) \xrightarrow{h} \operatorname{im}(h) \to 0\] with a retraction of the inclusion given by
		\begin{align*}
		r: \calE_k(G) & \to \Hom(G, k^\times)\\
		[C] &\mapsto \calH_C(1)
		\end{align*}
		Therefore, $\calE_k(G) \cong  \Hom(G, k^\times) \times \operatorname{im}(h)$. With this characterization, we view h-marks as analogues of the usual marks associated to elements of the Burnside ring.
	\end{remark}

    \begin{remark}
        Endotrivial complexes, after a possible shift in degree, are examples of \textit{endosplit $p$-permutation resolutions,} an object first defined by Rickard in \cite[Section 7]{R96} (note in \cite{R96}, they are called endosplit permutation resolutions). A chain complex $C \in Ch^b({}_{kG}\mathbf{mod})$ is an endosplit $p$-permutation resolution if $C$ has homology concentrated in degree 0, and $C^* \otimes_k C$ is split. It is easy to see that a chain complex is endotrivial if and only if it is (up to a shift) if it is an endosplit $p$-permutation resolution for a $kG$-module with $k$-dimension one. This observation was first communicated to the author by Markus Linckelmann.

        As a result, we obtain a lifting theorem to $\calO$.
    \end{remark}
    \begin{theorem}{\cite[Theorem 7.1]{R96}}
        Let $G$ be a finite group and let $M$ be a $kG$-module that has an endosplit $p$-permutation resolution $X_M$. Then $M$ can be lifted to a $\calO G$-module that has an endosplit $p$-permutation resolution over $\calO$, and this lift is unique up to isomorphism.
    \end{theorem}
    \begin{corollary}
        Let $C$ be an endotrivial $kG$-complex. There is a unique (up to isomorphism) complex $\widehat{C}$ of $p$-permutation $\calO G$-modules satisfying $\widehat{C}^* \otimes_\calO \widehat{C} \simeq \bigO$ and $k \otimes_\calO \widehat{C} \cong C$.
    \end{corollary}
    \begin{proof}
        This follows immediately from the previous theorem, since the shifted chain complex $C[-h_C(1)]$ is an endosplit $p$-permutation resolution.
    \end{proof}

	\section{Basic properties of endotrivial complexes} \label{basicpropsofendotrivs}

	Recall that $O(T(kG))$ denotes the orthogonal unit group of the trivial source ring. The following proposition ensures that $\calE_k(G)$ is compatible with $O(T(kG))$.

	\begin{prop}
		The following map is a well-defined a group homomorphism.
		\begin{align*}
		\Lambda: \calE_k(G) &\to O(T(kG))\\
		[C] &\mapsto \sum_{i \in \Z}(-1)^i C_i
		\end{align*}
	\end{prop}
	\begin{proof}
		To show that the map is well-defined, suppose $C\simeq C'$ are endotrivial complexes. Then there exist contractible complexes $D, D'$ such that $C \oplus D \simeq C' \oplus D'$, thus $\Lambda(C) + \Lambda(D) = \Lambda(C') + \Lambda(D') \in T(kG)$. Since every bounded contractible complex can be expressed as a finite direct sum of complexes of the form $0 \to M \to M \to 0$, $\Lambda(D) = \Lambda(D') = 0$, so $\Lambda(C) = \Lambda(C')$. A routine verification shows $\Lambda$ commutes with duals and tensor products. It follows that $\Lambda$ is a group homomorphism and  $\Lambda(C\otimes_k C^*) = \Lambda(C)\otimes_k \Lambda(C)^* = [k]$. Thus $\Lambda(C) \in O(T(kG))$.
	\end{proof}

	$\Lambda(C)$ is referred to as the \textit{Lefschetz invariant} of $C$. The image of $\Lambda$ will be of considerable interest in the sequel.

	\subsection{Relationships between h-marks and functorial constructions}

	\begin{remark}
		Any group isomorphism $f: G'\to G$ induces an algebra isomorphism $f: kG' \to kG$, hence a functor $\Iso_f: {}_{kG}\mathbf{mod} \to {}_{kG'}\mathbf{mod}$ which restricts to $\Iso_f: {}_{kG}\textbf{triv} \to {}_{kG'}\textbf{triv}$ and a group isomorphism $\Iso_f: \calE_k(G) \to \calE_k(G')$. Note $\Iso_f$ is the identity automorphism if $G= G'$ and $f$ is an inner automorphism of $G$.

        We define the group homomorphism $F$ as follows.

		\begin{align*}
		F: \prod_{P \in s_p(G)} \Z \times \Hom(N_G(P)/P, k^\times) &\to \prod_{P' \in s_p(G')} \Z \times \Hom(N_{G'}(P')/P', k^\times)\\
		(x_P, \rho_P)_{P\in s_p(G)} &\mapsto (x_{f(P')},  \rho_{f(P')}\circ f )_{P'\in s_p(G')}
		\end{align*}
        Then, it is routine to verify that the following diagram commutes.
		\begin{figure}[H]
			\centering
			\begin{tikzcd}
			\calE_k(G) \ar[r, "\Xi"] \ar[d, "\text{iso}_f"] & \prod_{P \in s_p(G)} \Z \times \Hom(N_G(P)/P, k^\times) \ar[d, "F"] \\
			\calE_k(G') \ar[r, "\Xi"] & \prod_{P' \in s_p(G')} \Z \times \Hom(N_{G'}(P')/P', k^\times)
			\end{tikzcd}
		\end{figure}
	\end{remark}

	By Proposition \ref{brauerpreservesendotriviality}, restriction to a subgroup, inflation from a quotient, and the Brauer construction all preserve endotrivial complexes as well. We next describe how the embedding $\Xi$ behaves with respect to endotriviality-preserving operations.

	\begin{prop}
		\begin{enumerate}
			\item Let $H \leq G$. The following diagram commutes, where $\pi$ is the projection map which applies restriction on group homomorphisms $\rho: N_G(P)/P \to k^\times$  via the inclusion $N_H(P)/P \hookrightarrow N_G(P)/P$.
			\begin{figure}[H]
				\centering
				\begin{tikzcd}
				\calE_k(G) \ar[r, "\Xi"] \ar[d, "\text{res}^G_H"] & \prod_{P \in s_p(G)} \Z \times \Hom(N_G(P)/P, k^\times) \ar[d, "\pi"] \\
				\calE_k(H) \ar[r, "\Xi"] & \prod_{P \in s_p(H)} \Z \times \Hom(N_H(P)/P, k^\times)
				\end{tikzcd}
			\end{figure}

			\item Let $Q \in s_p(G)$ be a $p$-subgroup of $G$, and regard $-(Q)$ as a functor ${}_{kG}\mathbf{mod} \to {}_{kN_G(Q)}\mathbf{mod}$ for ease of notation. Define the map $B$ by
			\begin{align*}
			B: \prod_{P \in s_p(G)} \Z \times \Hom(N_G(P)/P, k^\times) &\to  \prod_{P \in  s_p(N_G(Q)), Q\leq P} \Z \times \Hom(N_{N_G(Q)}(P)/P, k^\times) \\
			(x_P, \rho_P)_{P \in s_p(G)} &\mapsto
			\left(x_P, \rho_P|_{N_{N_G(Q)}(P)}\right)_{P \in s_p(N_G(Q)), Q \leq P}
			\end{align*}
			Then the following diagram commutes, where $\pi$ denotes projection:
			\begin{figure}[H]
				\centering
				\begin{tikzcd}
				\calE_k(G) \ar[r, "\Xi"] \ar[d, "-(Q)"]& \prod_{P \in s_p(G)} \Z \times \Hom(N_G(P)/P, k^\times) \ar[d, "B"] \\
				\calE_k(N_G(Q)) \ar[r, "\pi \circ \Xi"] &  \prod_{P \in s_p( N_G(Q)), Q \leq P} \Z \times \Hom(N_{N_G(Q)}(P)/P, k^\times)
				\end{tikzcd}
			\end{figure}

			\item Let $N \trianglelefteq G$. Define the map $I$ by
			\begin{align*}
			I: \prod_{P/N \in s_p(G/N)} \Z \times \Hom(N_{G/N}(P/N)/(P/N), k^\times) &\to  \prod_{P \in s_p(G)} \Z \times \Hom(N_G(P)/P,k^\times) \\
			(x_{P/N}, \rho_{P/N})_{P/N \in s_p(G/N)} &\mapsto \left(x_{PN/N},\Inf^{N_G(P)/P}_{N_{G/N}(PN/N)/(PN/N)}\rho_{PN/N}\right)_{P \in s_p(G)}
			\end{align*}
			Then the following diagram commutes:
			\begin{figure}[H]
				\centering
				\begin{tikzcd}
				\calE_k(G/N) \ar[r, "\Xi"] \ar[d, "\text{inf}^G_{G/N}"]& \prod_{P/N \in s_p(G/N)} \Z \times \Hom(N_{G/N}(P/N)/(P/N), k^\times)  \ar[d, "I"] \\
				\calE_k(G) \ar[r, "\Xi"] & (\prod_{P \in s_p(G)} \Z \times \Hom(N_G(P)/P,k^\times)
				\end{tikzcd}
			\end{figure}
		\end{enumerate}
		In particular, (a) and (c) along with the previous remark imply if $f: G' \to G$ is any group homomorphism, then the group homomorphism $\text{res}_f: \calE_k(G) \to \calE_k(G')$ induced by the restriction functor corresponding to $f$ preserves endotriviality.
	\end{prop}
	\begin{proof}
		(a) follows from Proposition \ref{brauercommuteswithdualsandconj} (b). Let $Q \leq P \leq N_G(Q)$ (so $Q \trianglelefteq P \leq G$), then (b) follows by Proposition \ref{brauerppermproperties} (d). Finally (c) follows by observing that we have a natural isomorphism of $G/N$-modules $(\Inf^G_{G/N} M)(P) \cong \Inf^{N_G(P)}_{N_{G/N}(PN/N)} (M(PN/N))$ for any $P \in s_p(G)$, where in this case, $-(P)$ is regarded as a functor ${}_{kG}\textbf{triv} \to {}_{kN_G(P)}\textbf{triv}$.
	\end{proof}

	\begin{remark}\label{trivsourceringdecomp}
		Given some endotrivial complex $C$, it is immediate what orthogonal unit it corresponds to, $\Lambda(C)$. However given only the h-marks $h_C$, the corresponding orthogonal unit is more difficult to determine.

		Boltje and Carman in \cite{BC23} determined a decomposition of $O(T(kG))$ which we denote by $\kappa$, \[\kappa: O(T(kG)) \xrightarrow{\cong} (B(S)^G)^\times \times \Hom(G, k^\times) \times \left(\prod_{P \in s_p(G)} \Hom(N_G(P)/PC_G(P), k^\times)\right)'.\] Here $(B(S)^G)^\times \leq B(S)^\times$, with  $S \in \text{Syl}_p(G)$, is defined as follows. Given $x \in B(S)^\times$, we have $x \in (B(S)^G)^\times$ if and only if $|x^P| = |x^Q|$ for all pairs $P,Q$ of $G$-conjugate subgroups of $S$. We refer the reader to \cite{BaC19} for details on $(B(S)^G)^\times$, and more generally for the Burnside ring of a fusion system.

		The third constituent is subject to the coherence condition \[\chi_P(xPC_G(P)) = \chi_{P\langle x_p\rangle}\big(x P\langle x_p\rangle C_G(P\langle x_p\rangle)\big)\] for all $P \in s_p(G)$ and $x \in G$. Here $x_p$ denotes the $p$-part of $x$; we have a unique decomposition $x = x_px_{p'} = x_{p'}x_p$, where $x_p$ has $p$-power order, and $x_{p'}$ has $p'$-order, and $P\langle x_p\rangle$ is the subgroup generated by $P$ and $x_p$. We equivalently consider the homomorphisms in the third constituent as elements of $\Hom(N_G(P), k^\times)$ whose kernel contains $PC_G(P)$. We denote this component by $\mathcal{L}_G$, for ``local'' homology. The middle constituent on the other hand corresponds to the ``global'' homology.

		We describe $\kappa$ in greater detail, following \cite{BC23}. First, \cite[Theorem A]{BC23} states that there is an injective homomorphism \[\beta_G: T(kG) \to \left(\prod_{P \in s_p(G)} R(K[N_G(P)/P])\right)^G,\] whose image consists of character tuples satisfying the coherence condition from before: for each $P \in s_p(G)$ and $x \in N_G(P)$, one has $\chi_P(xP) = \chi_{P\langle x_p\rangle} (xP\langle x_p\rangle)$. In particular,
		\[(\beta_G(x))_P = K\otimes_\bigO \widehat{x(P)} \in R(K[N_G(P)/P]),\] where $\widehat{(-)}$ denotes the isomorphism $T(\bigO G) \cong T(kG)$ induced by taking the unique lift of a $p$-permutation $kG$-module to a $p$-permutation $\bigO G$-module. We denote the subgroup of \textit{coherent} tuples satisfying this condition by \[\left(\prod_{P \in s_p(G)} R(K[N_G(P)/P])\right)'.\]
		Since the unit group of $R_K(G)$ is generated by virtual $k$-dimension one characters, for every orthogonal unit $u\in O(T(kG))$, there exist homomorphisms $\rho_P\in \Hom(N_G(P)/P, K^\times)$ and signs $\epsilon_P \in \{\pm 1\}$ such that  \[\beta_G(u) = (\epsilon_P\cdot \rho_P)_{P \in s_p(G)}.\]

		However, the coherence condition implies that if $x \in G$ is a $p$-element, $\rho_P(x) = 1$, so each $\rho_P$ descends to a degree one Brauer character, hence a homomorphism $\overline{\rho_P} \in \Hom(N_G(P)/P, k^\times)$. Therefore, $\beta_G$ may be regarded as a group homomorphism \[\beta_G: O(T(kG)) \to \prod_{P \in s_p(G)}\{\pm 1\} \times \Hom(N_G(P)/P, k^\times).\]

        In this sense, $\beta_G$ can be thought of as the trivial source ring analogue of $\Xi$. One may explicitly compute $\epsilon_P \cdot \rho_P$ by taking the image of $u(P)\in O(T(k[N_G(P)/P]))$ in $R_k(G)$ to obtain the degree one Brauer character.

		Now, for $S \in \Syl_p(G)$, there exist a sequence of homomorphisms \[(B(S)^G)^\times \hookrightarrow B(G)^\times \xrightarrow{k[-]} O(T(kG)) \xrightarrow{\Res^G_S} O(T(kS)) = T(kS)^\times \xrightarrow{\cong} B(S)^\times\] whose composition is the identity, giving a decomposition \[O(T(kG)) = (B(S)^G)^\times \times \ker(\operatorname{res}^G_S: O(T(kG)) \to O(T(kS)).\] It follows that the kernel is given precisely by units $u\in O(T(kG))$ for which $\beta_G(u) = (\rho_P)_{P\in s_p(G)}$ for group homomorphisms $\rho_P: N_G(P)/P \to k^\times$. This implies an isomorphism \[O(T(kG)) \cong (B(S)^G)^\times \times \left(\prod_{P\in s_p(G)} \Hom(N_G(P)/P, k^\times)\right)'.\]

		Now, we obtain a sequence of homomorphisms whose composition is the identity, \[\Hom(G,k^\times) \hookrightarrow \left(\prod_{P\in s_p(G)} \Hom(N_G(P)/P, k^\times)\right)' \xrightarrow{(\rho_P) \mapsto \rho_1} \Hom(G,k^\times).\]  Here the latter map has kernel given by coherent tuples $(\rho_P)_{P\in s_p(G)}$ for which $PC_G(P) \leq \ker (\rho_P)$, which is precisely $\calL_G$. Therefore, we obtain a second isomorphism \[\left(\prod_{P\in s_p(G)} \Hom(N_G(P)/P)\right)' \cong \Hom(G, k^\times) \times \left(\prod_{P\in s_p(G)} \Hom(N_G(P)/PC_G(P))\right)' = \Hom(G,k^\times) \times \calL_G.\] This data completely describes $\kappa$. All claims described here are presented in greater detail and proven in \cite{BC23}.
	\end{remark}

    Define the \textit{dimension homomorphism} $\dim: \Z \to \{\pm 1\}$ by $\dim(i) = (-1)^i$. 	Then, the next proposition follows from the previous discussion and the definition of $\Xi$.

	\begin{prop}
		The following diagram commutes:
		\begin{figure}[H]
			\centering
			\begin{tikzcd}
			\calE_k(G) \ar[r, "\Xi"] \ar[d, "\Lambda"] & \prod_{P\in s_p(G)} \big(\Hom(N_G(P)/P, k^\times) \times \Z \big) \ar[d, "\id\times \dim"]\\
			O(T(kG)) \ar[r, "\beta_G"] & \prod_{P\in s_p(G)} \big(\Hom(N_G(P)/P, k^\times) \times \{\pm 1\} \big)
			\end{tikzcd}
		\end{figure}
	\end{prop}

	Remark \ref{burnsideringmarkhominverse} described the inverse of the $\Q$-linearized mark homomorphism $\Q\mathfrak{m}: \Q B(G) \to \left(\prod_{H\leq G} \Q\right)^G$. Given the h-marks and the nonzero homology of an endotrivial complex $C$, we can read off parts of the corresponding orthogonal unit according to the decomposition described in Remark \ref{trivsourceringdecomp} as follows:

	\begin{prop}
		The following diagram commutes:
		\begin{figure}[H]
			\centering
			\begin{tikzcd}
			\calE_k(G) \ar[r, "h"] \ar[d, "\Lambda"] &\left(\prod_{P \in s_p(G)} \Z\right)^G \ar[d, "\phi "] \\
			O(T(kG)) \ar[r, "\operatorname{res}^G_S "]&  \Q (B(S)^G)^\times
			\end{tikzcd}
		\end{figure}
		Here, for $S$ a Sylow $p$-subgroup of $G$, $\phi: \left(\prod_{P\in s_p(G)}\Z\right)^G \to \Q (B(S)^G)^\times$ is:
		\[\left(\prod_{P\in s_p(G)} \Z\right)^G \xrightarrow{\prod \dim}  \left(\prod_{P\leq S}\{\pm 1\} \right)^G \xrightarrow{\Q\mathfrak{m}\inv} \Q (B(S)^G)^\times\]

		In particular, the images of the two composites of maps in the original commutative diagram lay in $(B(S)^G)^\times$.
	\end{prop}
	\begin{proof}
		We have a diagram as follows:
		\begin{figure}[H]
			\centering
			\begin{tikzcd}
			\calE_k(S) \ar[dr, "h"]\\
			\calE_k(G) \ar[r, "h"] \ar[d, "\Lambda"] \ar[u, "\operatorname{res}^G_S"]&  \left(\prod_{P \in s_p(G)} \Z\right)^G \ar[d, " \phi"] \ar[rd, "\prod\dim"]  \\
			O(T(kG)) \ar[r, " \operatorname{res}^G_S"]& \Q (B(S)^G)^\times \ar[r, "\Q\mathfrak{m}"]  & \left(\prod_{P\leq S} \{\pm 1\}\right)^G
			\end{tikzcd}
		\end{figure}
		By construction of $\phi$, the right-most triangle commutes, and the top triangle commutes by properties of $h$ and restriction. Since $\Q\mathfrak{m}$ is an isomorphism, it suffices to show the trapezoid on the bottom row commutes, however this follows from the previous proposition.
	\end{proof}

	\subsection{Endotrivial complexes induce splendid autoequivalences}

	The goal of this section is to prove the following:

	\begin{theorem} \label{srcinductionthm}
		If $C$ is a endotrivial complex, then $\Ind^{G\times G}_{\Delta G}C$ is a splendid Rickard autoequivalence of $kG$, with $\Lambda(\Ind^{G\times G}_{\Delta G} C) = \Ind^{G\times G}_{\Delta G} (\Lambda(C))$.
	\end{theorem}

	We first review some necessary definitions. Here, we relax the definition of splendid Rickard complexes defined by Rickard in \cite{R96}, by allowing algebras to be direct summands of group algebras rather than block algebras.

	\begin{definition}\label{SRCdef}
		Let $A$ and $B$ be direct summands of group algebras $kG$ and $kH$ respectively. We say a chain complex $\Gamma$ of $(A,B)$-bimodules is a \textit{splendid Rickard complex between $A$ and $B$} if the following hold:
		\begin{enumerate}
			\item $\Gamma \otimes_B \Gamma^* \simeq A$.
			\item $\Gamma^* \otimes_A \Gamma \simeq B$.
			\item Each component of $\Gamma$ is a $p$-permutation module when viewed as an $(A\otimes_k B^{op})$-module, and each indecomposable constituent has a \textit{twisted diagonal} vertex, i.e. a subgroup of $G\times H$ of the form $\Delta(P, \phi, Q) = \{(\phi(g),g) \mid g \in Q\}$ for some $P \in s_p(G)$, $Q \in s_p(H)$, and $\phi: Q \to P$ an isomorphism.
		\end{enumerate}
		We say $\Gamma$ induces a \textit{splendid (derived) equivalence} $D^b({}_A\mathbf{mod}) \cong D^b({}_B\mathbf{mod})$, which is induced by the functors by $\Gamma\otimes_B -$ and $\Gamma^*\otimes_A -$. In fact, these functors induce equivalences $K^b({}_A\mathbf{mod}) \cong K^b({}_B\mathbf{mod})$ as well.

		It is easy to see that the set of all splendid Rickard autoequivalences of $A$ modulo homotopy equivalence forms a group with addition induced by $\otimes_A$. We denote this group by $\calS_k(A)$.
	\end{definition}

	$p$-permutation equivalences, first defined in \cite{BX07} by Boltje and Xu, can be viewed as an analogue of splendid Rickard complexes on a representation ring level. We state the more general definition given in \cite{BP20}, which no longer assumes shared Sylow subgroups.

	\begin{definition}\label{ppermdef}
		Let $A$ and $B$ be direct summands of group algebras $kG$ and $kH$ respectively. Write $T(A,B)$ for the Grothendieck group of $p$-permutation (when viewed as an $(A\otimes_k B^{op})$-module) $(A,B)$-bimodules. Let $C$ be a direct summand of the group algebra $kK$. The tensor product $\otimes_B$ induces a bilinear map $\cdot_B: T(A,B) \times T(B,C) \to T(A,C)$. Denote by $T^\Delta(A,B)$ the subgroup of $T(A,B)$ generated by bimodules with twisted diagonal vertices.

		We say $\gamma \in T^\Delta(A,B)$ is a \textit{$p$-permutation equivalence} if
		\[\gamma \cdot_A\gamma^* = [B]\in T^\Delta(B,B) \text{ and } \gamma^* \cdot_B \gamma = [A] \in T^\Delta(A,A).\]

		In this case, $\gamma$ induces a group isomorphism $T(A) \cong T(B)$ via the homomorphisms $\gamma \cdot_B -$ and $\gamma^* \cdot_A -$. Denote the set of $p$-permutation equivalences between $A$ and $B$ by $O(T^\Delta(A,B))$. If $A = B$, it is easy to see that  $O(T^\Delta(A,A))$ forms a group with addition induced by $\cdot_A$.
	\end{definition}

	The composite of the functors $\Iso_G^{\Delta G}: {}_{kG}\mathbf{mod} \to {}_{k[\Delta G]}\mathbf{mod},$ where $\Delta G$ is the diagonal subgroup $\Delta G = \{(g,g) : g\in G\},$ induction $\Ind^{G\times G}_{\Delta G}: {}_{k[\Delta G]}\mathbf{mod} \to {}_{k[G\times G]}\mathbf{mod},$ and the equivalence of categories given by the identification $ {}_{k[G\times G]}\mathbf{mod} \cong {}_{kG}\mathbf{mod}_{kG}$ via the group action $g\cdot m\cdot h := (g,h\inv) \cdot m$ will be abusively denoted $\Ind^{G\times G}_{\Delta G}: {}_{kG}\mathbf{mod} \to {}_{kG}\mathbf{mod}_{kG}$ when the context is clear. We will show that this functor transfers the necessary properties of endotriviality to splendor. The next two propositions are elementary. Note that a natural isomorphism of functors on preadditive categories extends to a natural isomorphism of functors on their chain complex categories as well.

	\begin{prop} \label{inddualcommute}
		Let $H \leq G$. The contravariant composite of functors $\Ind^G_H \circ (-)^*$, $(-)^* \circ \Ind^G_H: {}_{kH}\mathbf{mod} \to {}_{kG}\mathbf{mod}$ are naturally isomorphic.
	\end{prop}
	\begin{proof}
		Let $M$ be a $kH$-module. By the Yoneda embedding, we have a isomorphism $\Ind^G_H (M^*) \cong (\Ind^G_H M)^*$ natural in $M$ if and only if the functors $\Hom_{kG}(-, \Ind^G_H (M^*))$ and $\Hom_{kG}(-, (\Ind^G_H M)^*)$ are naturally isomorphic as functors, and this isomorphism is natural in $M$. Let $N$ be any $kG$-module, and denote the trivial $kG$-module by $k_G$ for clarity. Then, using the tensor-hom adjunction, the Frobenius property $\Ind^G_H(\Res^G_H V \otimes_k  W) \cong  V \otimes_k \Ind^G_H W$, and the biadjunction between induction and restriction (which are all natural in both arguments) yields the following sequence of natural isomorphisms.
		\begin{align*}
		\Hom_{kG}(N, \Ind^G_H(M^*))&\cong \Hom_{kH}(\Res^G_H N, M^*)\\
		&\cong \Hom_{kH}(M \otimes_k (\Res^G_H N), \Res^G_H k_G)\\
		&\cong \Hom_{kG}(\Ind^G_H(M \otimes_k \Res^G_H N), k_G)\\
		&\cong \Hom_{kG}((\Ind^G_H M)\otimes_k N, k_G)\\
		&\cong \Hom_{kG}(N, (\Ind^G_H M)^*)
		\end{align*}
	\end{proof}

	\begin{prop}
		The bifunctors $\Ind^{G\times G}_{\Delta G}(-)\otimes_{kG} \Ind^{G\times G}_{\Delta G}(-)$ and $\Ind^{G\times G}_{\Delta G}(- \otimes_k -): {}_{kG}\mathbf{mod}\times {}_{kG}\mathbf{mod} \to {}_{kG}\mathbf{mod}_{kG} $ are naturally isomorphic in both arguments.
	\end{prop}
    \begin{proof}
        See \cite[Corollary 2.4.13]{L181}.
    \end{proof}

	\begin{prop}\label{prop:dualidentification}
		The following diagram commutes up to natural isomorphism, where the vertical functors are bimodule identification and $\Ind_{\Delta G}^{G\times G}$ is regarded as usual:
		\begin{figure}[H]
			\centering
			\begin{tikzcd}
			{}_{k\Delta G}\mathbf{mod} \ar[rd, "\Ind^{G\times G}_{\Delta G}"]\\
			& {}_{k[G\times G]}\mathbf{mod} \ar[r, "(-)^*"] \ar[d] & {}_{k[G\times G]}\mathbf{mod} \ar[d] \\
			& {}_{kG}\mathbf{mod}_{kG} \ar[r, "(-)^*"] & {}_{kG}\mathbf{mod}_{kG}
			\end{tikzcd}
		\end{figure}
	\end{prop}
	\begin{proof}
		We construct an isomorphism $\phi: (\Ind^{G\times G}_{\Delta G} M)^*_1 \to (\Ind^{G\times G}_{\Delta G} M)^*_2$, where $(\Ind^{G\times G}_{\Delta G} M)^*_1$ corresponds to the top right composite, that is, for $a,b \in G$ it has actions defined by:
		\begin{align*}
		a\cdot f\big((g_1,g_2)\otimes m\big) \cdot b &= (a,b\inv) \cdot f\big((g_1,g_2)\otimes m\big)\\
		&= f\big((a\inv, b)(g_1,g_2)\otimes m\big)\\
		&= f\big((a\inv g_1,bg_2)\otimes m\big),
		\end{align*}
		and $(\Ind^{G\times G}_{\Delta G} M)^*_2$ corresponds to the bottom left composite, that is, it has actions defined by:
		\begin{align*}
		a\cdot f\big((g_1,g_2)\otimes m\big) \cdot b &= f\big(b\cdot \big((g_1,g_2)\otimes m\big)\cdot a\big)\\
		&= f\big((bg_1,a\inv g_2)\otimes m\big).
		\end{align*}

        We define $\phi: (\Ind^{G\times G}_{\Delta G} M)^*_1 \cong (\Ind^{G\times G}_{\Delta G}M)^*_2$ as follows:

        \[\phi: f \mapsto \big((g_1,g_2)\otimes m \mapsto f((g_2,g_1)\otimes m)\big)\]

        Above, $(g_2,g_1)\otimes m$ as above is considered an element in $(\Ind^{G\times G}_{\Delta G} M)^*_1$. We first check $\phi$ is well-defined with respect to the tensor product. Let $f \in (\Ind^{G\times G}_{\Delta G} M)^*_1$ and $(g_1,g_2)\otimes m \in \Ind^{G\times G}_{\Delta G}M$, regarded as $(kG,kG)$-bimodule. We have for any $g \in G$ that \[(g_1,g_2)\otimes m = (g_1g\inv, g_2g\inv)\otimes gm.\]
        Then for any $f \in (\Ind^{G\times G}_{\Delta G} M)^*_1$,
        \begin{align*}
            \phi(f)((g_1,g_2)\otimes m)&= \phi(f)((g_1g\inv, g_2g\inv)\otimes gm)\\
            &= f((g_2g\inv, (g_1g\inv)\otimes gm)\\
            &= f((g_2, g_1)\otimes m)\\
            &=  \phi(f)((g_1,g_2)\otimes m)
        \end{align*}
        Thus $\phi$ is well-defined. We next check that it is a $(kG,kG)$-bimodule homomorphism. Let $a,b \in G$.
        \begin{align*}
            \phi(a\cdot f\cdot b)((g_1, g_2)\otimes m) &= \phi (f((a\inv, b)\cdot -))((g_1, g_2)\otimes m)\\
            &= f((a\inv g_2, bg_1)\otimes m)\\
            &= \phi(f)((bg_1, a\inv g_2)\otimes m)\\
            &=(a\cdot \phi(f)\cdot b)((g_1, g_2)\otimes m)
        \end{align*}
        Therefore, $\phi$ is a $(kG,kG)$-bimodule homomorphism. It is clear $\phi$ is bijective, thus it is an isomorphism as desired.

        To see $\phi$ is natural, first note that all morphisms which arise are of the form $(\id \otimes f)^*$ for $f: M\to N$ any left $kG$-module homomorphism. Then, it is straightforward to check the following diagram commutes:
		\begin{figure}[H]
			\centering
			\begin{tikzcd}
			(k[G\times G]\otimes_{k\Delta G} M)^*_1 \ar[r, "\phi_M"] & (k[G\times G]\otimes_{k\Delta G} M)^*_2\\
			(k[G\times G]\otimes_{k\Delta G} N)^*_1 \ar[r, "\phi_N"] \ar[u, "(\id \otimes f)^*"] & (k[G\times G]\otimes_{k\Delta G} N)^*_2 \ar[u, "(\id \otimes f)^*"]
			\end{tikzcd}
		\end{figure}
	\end{proof}

	\begin{proof}[Proof of Theorem \ref{srcinductionthm}]
        Propositions \ref{inddualcommute} and \ref{prop:dualidentification} imply the following diagram is commutative.

		\begin{figure}[H]
			\centering
			\begin{tikzcd}
			{}_{k[\Delta G]}\mathbf{mod} \ar[r, "(-)^*"] \ar[d, "\Ind^{G\times G}_{\Delta G}"] & {}_{k[\Delta G]}\mathbf{mod} \ar[d, "\Ind^{G\times G}_{\Delta G}"]\\
			{}_{k[G\times G]}\mathbf{mod} \ar[r, "(-)^*"] \ar[d, "\cong"] & {}_{k[G\times G]}\mathbf{mod}  \ar[d, "\cong"]\\
			{}_{kG}\mathbf{mod}_{kG} \ar[r, "(-)^*"] & {}_{kG}\mathbf{mod}_{kG}
			\end{tikzcd}
		\end{figure}
		Therefore, $\Ind^{G\times G}_{\Delta G}(C^*) \cong \left(\Ind^{G\times G}_{\Delta G}C\right)^*$, and
		\[kG\cong\Ind^{G\times G}_{\Delta G} (k)\simeq \Ind^{G\times G}_{\Delta G}(C \otimes_{k}C^*) \cong\Ind^{G\times G}_{\Delta G}C\otimes_{kG}\Ind^{G\times G}_{\Delta G}(C^*) \cong \Ind^{G\times G}_{\Delta G}(C)\otimes_{kG}\big(\Ind^{G\times G}_{\Delta G}C\big)^*.\]
		If $C$ is an endotrivial complex, then the components of $C$ are $p$-permutation modules with vertices contained in $\Delta G$, since for any subgroup $H \leq G$, $\Ind^{G\times G}_{\Delta G}\circ \Iso^{\Delta G}_G \circ \Ind^G_H = \Ind^{G\times G}_{\Delta H}\circ \Iso^{\Delta H}_H.$ Therefore, $\Ind^{G\times G}_{\Delta G}C \otimes_{kG}\Ind^{G\times G}_{\Delta G}(C)^* \simeq kG$, so $\Ind^{G\times G}_{\Delta G}(C)$ is a splendid Rickard autoequivalence of $kG$. The second homotopy equivalence follows similarly. The final statement follows from additivity of induction.
	\end{proof}

     \begin{remark}
         In particular, if $C$ is a lift of an orthogonal unit $u \in O(T(kG))$, then $\Ind^{G\times G}_{\Delta G} C$ is a splendid lift of the $p$-permutation autoequivalence $\Ind^{G\times G}_{\Delta G} u \in O(T^\Delta(kG, kG))$. Moreover, $\Ind^{G\times G}_{\Delta G}$ reflects isomorphisms (which we show in the following lemma), implying each unique (up to isomorphism) endotrivial complex defines an corresponding unique (up to isomorphism) splendid autoequivalence. Therefore, we obtain an injective group homomorphism $\calE_k(G) \to \calS_k(kG)$.
     \end{remark}

    \begin{lemma}
    	Let $M$ be a $kG$-module. Then we have a natural isomorphism $(\Ind^{G\times G}_{\Delta G} M)^{1\times G} \cong M$, where we identify $(\Ind^{G\times G}_{\Delta G} M)^{1\times G}$ as a $kG$-module by the isomorphism $(G\times G)/(1\times G)\cong G$. In particular, if $C_1,C_2$ are chain complexes of $kG$-modules satisfying $\Ind^{G\times G}_{\Delta G}C_1 \cong \Ind^{G\times G}_{\Delta G} C_2,$ then $C_1\cong C_2$.
    \end{lemma}
    \begin{proof}
    	Observe every element of $(\Ind^{G\times G}_{\Delta G} M)^{1\times G}$ can be expressed as an $k$-linear combination of elements of the form \[\left(g, \sum_{g' \in G}g'\right) \otimes m, \quad g\in G,\text{ and } m \in M.\] Then, it is straightforward to verify that the homomorphism $\phi: (\Ind^{G\times G}_{\Delta G} M)^{1\times G} \to M$ induced by \[\left(g, \sum_{g' \in G}g'\right) \otimes m \mapsto gm\] is a well-defined natural isomorphism.
    \end{proof}

	\section{The faithful constituent of $\calE_k(G)$} \label{faithfulsection}

	Bouc introduced a theory of ``biset functors,'' an abstraction of constructions which have an abelian group associated to to every finite group, and corresponding induction, restriction, transfer, inflation, and deflation maps between groups. This can be viewed as an extension of global Mackey functors. This approach has led to a number of results, such as describing the Dade group $D(P)$ of a finite $p$-group $P$, which parameterizes the ``capped'' endotrivial $kP$-modules. $D(P)$ is connected to the study of endotrivial modules, as there is an embedding $\calT(P) \hookrightarrow D(P)$ given by sending an equivalence class $[M] \in \calT(P)$ to its corresponding equivalence class $[M] \in D(P)$. We refer the reader to Bouc's text on biset functors \cite[Chapter 12]{Bou10} for further details.

	Following Bouc, we define the notion of a faithful endotrivial complex which one may think of as an endotrivial complex which contains no parts which are inflated from a proper quotient group. The set of all of these forms the faithful subgroup of $\calE_k(G)$. For ease of notation, we drop bracket notation when referring to elements of $\calE_k(G)$. The following definition is adapted from \cite[Chapter 6]{Bou10} with a few modifications. Let  $s_p^\triangleleft(G)$ denote the set of all normal $p$-subgroups of $G$.

	\begin{definition}\label{faithfuldef}
		Let $P \in s_p^\triangleleft(G)$. Then $\Inf^G_{G/P}$ induces an injective group homomorphism $\calE_k(G/P) \to \calE_k(G)$. This homomorphism is split injective with retraction given by $\Def^G_{G/P} := (-)(P)$. Define the \textit{faithful component} of $\calE_k(G)$, $\partial \calE_k(G)$ as follows.
		\[\partial \calE_k(G) := \bigcap_{1 < P \in s_p^\triangleleft(G)} \ker \Def^G_{G/P}.\]

		For example, when $G$ has only normal $p$-subgroups, $\partial \calE_k(G)$ consists of endotrivial complexes whose h-marks are all zero and local homology is trivial, except possibly at the trivial subgroup. In that case, it follows that $\partial\calE_k(G)$ has $\Z$-rank at most 1, corresponding to the h-mark at the trivial subgroup. Call any $C \in \partial \calE_k(G)$ a \textit{faithful} endotrivial complex.
	\end{definition}

	\begin{remark}
		$\calE_k$ may be regarded as a ``partial'' biset functor, in that the biset operations which are permitted are inflation, isomorphism, and restriction, and deflation only for normal $p$-subgroups of a fixed finite group. In this way, the faithful component defined here is analogous to the faithful component of a biset functor as defined in \cite{Bou10}.  However, there is no known induction - usual induction of chain complexes is not multiplicative, and tensor induction of chain complexes (see \cite{E61} for details) does not in general preserve endotriviality.

		For example, let the Klein-4 group $V_4 \cong C_2 \times C_2$ have generators $\sigma, \tau \in V_4$, and $C$ be the endotrivial complex of $k\langle\sigma\rangle $-modules given by $k\langle \sigma \rangle \to k$ with the differential the augmentation homomorphism $k$ in degree zero. Let $D = \TI^{V_4}_{\langle\sigma\rangle} C$. Then, it is routine to show that \[ D_2 \cong  k[V_4/\langle \tau\rangle] \oplus k[V_4/\langle \sigma\tau\rangle], D_1 \cong kV_4, D_0 \cong k.\] $D$ cannot be an endotrivial complex, since \[\big(D(\langle \tau\rangle)\big)_2 \cong k[V_4/\langle\tau\rangle], \big(D(\langle \tau\rangle)\big)_1 \cong 0, \big(D(\langle \tau\rangle)\big)_0\cong k\] which cannot possibly have homology concentrated in one degree.

		However, we may view $\calE_k$ as a biset functor after restricting to the subcategory $\Z$-linearly generated by all restriction, inflation, and transfer bisets, and deflation bisets only of the form $\Def^G_{G/P}$ for normal $p$-subgroups $P \in s_p^\triangleleft(G)$.
	\end{remark}

	\begin{example}
		The following examples come from Proposition \ref{prk1gen}. If $G = C_{p^n}$ with $p > 2$ or $p = 2$ and $n > 1$, then $\partial\calE_k(G)$ is generated by the endotrivial complex generated by truncating the period 2 free resolution of $k$. Say $G = \langle\sigma\rangle$, then the endotrivial complex $C$ is as follows: \[C = \big(kG \xrightarrow{d_2} kG\xrightarrow{d_1} k, \quad d_2: \sigma \mapsto \sigma - 1, \quad d_1: \sigma \mapsto 1\big).\] Indeed, $h(C(1)) = 2$ and $C(P) \cong k$ for any $1 < P \in s_p(G)$.

		If $p=2$ and $G = C_2$, then $\partial\calE_k(G)$ is generated by the endotrivial complex generated by truncating the period 1 free resolution of $k$, \[C = \big(kC_2 \to k,\quad \sigma\mapsto 1\big).\] Finally if $G$ is any group of order prime to $p$ or does not contain any nontrivial normal $p$-subgroups, then vacuously, $\partial \calE_k(G) = \calE_k(G)$.
	\end{example}

	Computing faithful endotrivial complexes which generate the faithful constituent will be the main focus of Section \ref{section5}, since as the next proposition will imply, determining $\partial\calE_k(G)$ is, assuming an inductive hypothesis, the only necessary information to completely determine $\calE_k(G)$. The next theorem and proof are adapted from \cite[Proposition 6.3.3]{Bou10}, but reformulated to be presentable in a self-contained manner.

	\begin{theorem}{\cite[Proposition 6.3.3]{Bou10}}\label{faithfuldecomp}
		Let $G$ be any group. Define the following group homomorphism:
		\begin{align*}
		\Phi: \calE_k(G) &\to \prod_{P \in s_p^\triangleleft(G)}\calE_k(G/P)\\
		C &\mapsto  \left(\sum_{P\leq Q \in s_p^\triangleleft(G)}\mu(P,Q)\cdot \left(\Inf^{G/P}_{G/Q}C(Q)\right)\right)_{P \in s^\triangleleft_p(G)}
		\end{align*}
		Here, $\mu$ is the M\"obius function associated to the poset of normal $p$-subgroups of $G$ $s_p^\triangleleft (G)$.
		\begin{enumerate}
			\item The image of $\Phi$ is contained in $\prod_{P\in s_p^\triangleleft(G)}\partial\calE_k(G/P)$.
			\item $\Phi: \calE_k(G) \to \prod_{P \in s_p^\triangleleft(G)}\partial\calE_k(G/P)$ is an isomorphism of groups, with inverse given by
			\begin{align*}
			\Psi: \prod_{P\in s_p^\triangleleft(G)}\partial\calE_k(G/P) &\to \calE_k(G) \\
			(C_P)_{P\in s_p^\triangleleft(G)} & \mapsto \sum_{P\in s_p^\triangleleft(G)}\Inf^G_{G/P} C_P
			\end{align*}
		\end{enumerate}
	\end{theorem}

	\begin{proof}
		Set $\Phi_S$ to be the component of $\Phi$ at $S$, that is, $\Phi = (\Phi_S)_{S\in s_p^\triangleleft(G)}$. To show (a) it suffices to show that $\big(\im \Phi_S\big)(P) = k$ when $P > S$. Fix an endotrivial complex of $kG$-modules $C$. We first exhibit a clever reindexing: for fixed $S$ and $P \geq S$ with $P, S \in s_p^\triangleleft(G)$,

		\begin{align*}
		\left(\sum_{S\leq Q \in s_p^\triangleleft(G)} \mu(S, Q)\cdot \big(\Inf^{G/S}_{G/Q}C(Q)\big)\right)(P) &= \sum_{Q \in s_p^\triangleleft(G)}  \mu(S, Q)\big(\Inf^G_{G/Q}C(Q)\big)(P)\\
		&= \sum_{S\leq Q \in s_p^\triangleleft(G)} \mu(S, Q)\cdot \left( \Inf^{G/P}_{G/PQ} C(PQ)\right)\\
		&= \sum_{SP\leq X \in s_p^\triangleleft(G)}\left(\sum_{S\leq Q \in s_p^\triangleleft(G), X = PQ} \mu(S, Q)\right) \cdot \left( \Inf^{G/P}_{G/X} C(X)\right)\\
		\end{align*}

		Set $s_X = \sum_{S\leq Q\in s_p^\triangleleft(G), PQ = X} \mu(S, Q)$, then it suffices to show $s_X = 0$ unless $PS = S$. If $PS \neq S$, then \[s_{PS} = \sum_{Q \in s_p^\triangleleft(G), S \leq Q \leq PS} \mu(S,Q) = 0.\] Then, for $PS \leq Y \in s_p^\triangleleft(G)$, \[\sum_{X \in s_p^\triangleleft(G), PS \leq X \leq Y }s_X = \sum_{X \in s_p^\triangleleft(G), PS \leq X \leq Y } \sum_{S\leq Q \in s_p^\triangleleft(G), PQ = X} \mu(S, Q) = \sum_{Q\in s_p^\triangleleft(G), S \leq Q \leq Y} \mu(S, Q) = 0,\] and inducting on the poset $s_P^\triangleleft(G)$ allows us to conclude $s_Y = 0$. Thus, the exponent is zero unless $ P = S$, and we conclude $\im \Phi_P \subseteq \partial \calE_k(G/P)$.

		For (b), by M\"{o}bius inversion, it follows that for any endotrivial complex of $kG$-modules $C$, \[\sum_{P, Q\in s_p^\triangleleft(G), P \leq Q}\Inf^{G/P}_{G/Q}\Phi_Q(C) = C(P).\] In particular, $\sum_{P\in s_p^\triangleleft(G)} \inf^G_{G/P} \Phi_P(C) = C$, which demonstrates $\Psi \circ\Phi = \id_{\calE_k(G)}$. To show that $\Phi$ and $\Psi$ are inverses, it suffices to show that $\Psi$ is injective.

		Suppose for contradiction that $\ker \Psi \neq (k)_{P\in s_p(G)}$, and say nontrivial $(C_P)_{P\in s_p^\triangleleft(G)} \in \ker\Psi$. Then the product of all $\Inf^G_{G/P} C_P$ is the trivial complex. There must exist a locally maximal $X \in s_p^\triangleleft(G)$ with respect to the property that $C_X$ is a nontrivial faithful endotrivial complex. First, assume $X$ is not a maximal element of $s_p^\triangleleft(G)$. Then, $C_X$ has unique highest nonzero h-mark at a subgroup $X' \geq X$ with respect to subgroup order, where $X'$ does not contain as subgroup any normal subgroups containing $X$ besides $X$ itself. But then, it is impossible that $h\left(\left(\sum_{P\in s_p^\triangleleft(G)}\inf^G_{G/P} C_P\right)(X)\right) = 0$, since for all other $X < Y \in s_p^\triangleleft(G)$, $h(C_Y(X)) = 0$ by maximality of $X$, and for all other $X \not\leq Y \in s_p^\triangleleft(G)$, $h(C_Y(X)) = 0$ since $C_Y$ is a faithful endotrivial $k[G/Y]$-complex.

		Otherwise, if $X$ is maximal, it might also be the case that $C_X$ has h-marks 0 everywhere, in which case $C_X = k_\omega$ for some nontrivial $\omega \in \Hom(G/X, k^\times)$. In this case, $X$ is a global maximum by basic group-theoretic arguments. It follows that $C_X$ is the only chain complex in the tuple $(C_P)$ for which $(C_P)(X) \neq k$, by faithfulness. Therefore \[\sum_{P \in s_p^\triangleleft(G)} \Inf^{G}_{G/P} (C_P) (X) = \left(\sum_{P \in s_p^\triangleleft(G)} \Inf^{G}_{G/P} C_P\right)(X) \neq k.\]
		Thus $(C_P)$ cannot exist, and $\ker \Psi = (k)_{P\leq G}$, as desired.
	\end{proof}

	\begin{remark}
		Since $\calE_k(G)$ decomposes into a direct product of faithful components, to completely determine the structure of $\calE_k(G)$, it suffices to determine $\partial\calE_k(G)$, assuming we have already determined $\calE_k(G)$ for all groups of smaller order. One may ask what other restrictions can be placed upon elements of $\partial\calE_k(G)$.

		We will show that if $[C] \in\partial\calE_k(G)$, then there is a representative $C$ for with all of its components ``faithful'' as well, in that after applying the Brauer construction at any nontrivial normal subgroup, each of its components vanish with the exception of a unique module of $k$-dimension one in one degree. Explicitly, if $\calX$ is the set of $p$-subgroups of $G$ which do not contain a nontrivial normal subgroup of $G$, then there exists some $i \in \Z$ such that for all $j \neq i$, $C_j$ is $\calX$-projective, and $C_i$ contains one indecomposable summand which has vertex $S \in \Syl_p(G)$, and all other summands have vertex contained in $\calX$. To do this, we first prove a generalization of \cite[Theorem 7.9]{Bou98}.
	\end{remark}

	\begin{theorem}\label{boucthmgeneralization}
		Let $C$ be a bounded chain complex of $p$-permutation $kG$-modules, and let $\mathcal{X}$ be a nonempty subset of $s_p(G)$ which is closed under $G$-conjugation and taking subgroups. The following are equivalent:
		\begin{enumerate}
			\item For all $P \not\in\mathcal{X}$, $C(P)$ is acyclic.
			\item There exists a chain complex $D$ with $C\simeq D$ such that for all $i \in \Z$, $D_i$ is $\calX$-projective.
		\end{enumerate}
		In particular, if $\mathcal{X} = \{1\}$, we obtain \cite[Theorem 7.9]{Bou98}.
	\end{theorem}

	To prove this, we first state a lemma which refines Theorem \ref{ppermvertex}.

	\begin{lemma}
		Let $M,N_1,\dots, N_l$ be $p$-permutation $kG$-modules, with each $N_i$ indecomposable with vertex $P_i$, and let $f: M\to N_1 \oplus\cdots \oplus N_l$ be a $kG$-module homomorphism. $f$ is split surjective if and only if the $kN_G(P_i)/P_i$-module homomorphism $f(P_i)$ is surjective for all $i \in \{1,\dots, l\}$.
	\end{lemma}
	\begin{proof}
		The forward direction is trivial. We proceed by induction on $l$. Note the base case $l = 1$ is Theorem \ref{ppermvertex}.

		Suppose $f(P_i)$ is surjective for all $i \in \{1,\dots, l\}$. Since $f(P_1)$ is surjective, Theorem \ref{ppermvertex} implies there exists a direct summand $M_1$ of $M$ isomorphic to $N_1$ such that $f(M_1) = N_1$ and $f|_{M_1}$ is an isomorphism. Let $M = M_1 \oplus  M'$, where $M' = \ker f|_{M_1}$. After a possible change of complement of $N_1$, it follows that $f|_{M'}(P_2)$ surjects onto $N_2(P_2) \oplus \cdots \oplus N_l(P_2)$, since $\im f(P_2) = (N_1 \oplus \cdots \oplus N_l)(P_2)$, and the inductive hypothesis completes the proof.
	\end{proof}

	\begin{proof}[Proof of Theorem \ref{boucthmgeneralization}]
	    (b) implies (a) is straightforward. We show (a) implies (b). Let $n$ be the minimum integer for which $C_n \neq 0$. For each nonzero $C_n$, write $C_n = X_n \oplus Y_n$, where $X_n$ consists of all indecomposable summands with vertex contained in $\mathcal{X}$, so $Y_n$ consists of all indecomposable summands with vertex not contained in $\mathcal{X}$. We construct a chain complex $D$ identical to $C$, except in degree $n+1$, where we set $D_{n+1} = X_n \oplus C_{n+1}$, and add the identity map on $X_n$  to the differential of $D_{n+1}$, $d_{n+1}$. Since $C(Q)$ is acyclic for all $Q \not \in \calX$ by assumption, so is $D(Q)$, as $X_n(Q) = 0$ by Proposition \ref{brauerppermproperties}. Therefore $d_{n+1}$ composed with projection onto $Y_n$ is split surjective by the previous lemma, and by construction, it follows that $d_{n+1}$ is split surjective. Thus, we have that $D \cong D' \oplus (C_n \xrightarrow{\cong} C_n)$, where $D'$ has lowest nonzero degree $n+1$.

		Note that $C$ is homotopy equivalent to a shift of the mapping cone of the chain map $D' \to X_n$. Moreover, if $Q \not\in \calX$, $C(Q) \cong D(Q) \simeq D'(Q)$, and these complexes are by assumption acyclic. We now perform an inductive argument as follows. If the length of $C$ is one, then (b) holds, since there is only one nonzero term which must vanish after applying the Brauer construction at all $P \not\in \calX$, and the rest follows by Proposition \ref{brauerppermproperties} (c).

        Now we assume (a) implies (b) holds for complexes of length at most $n$, and suppose $C$ has length $n+1$. If the length of $C$ is $n+1$, then $D'$ has length $n$. Since $D'(P)$ is also acyclic for all $P \not\in \calX$, $D'$ contains only $\calX$-projective modules. Since $C$ is homotopy equivalent to a shift of the mapping cone $D' \to X_n[n]$, and both $D'$ and $X_n[n]$ are homotopy equivalent to complexes with only $\calX$-projective modules, we conclude $C$ is also homotopy equivalent to a complex with only $\calX$-projective modules.
	\end{proof}

     \begin{lemma}\label{relprojectivitylimitingthm}
        Let $C$ be a bounded complex of $p$-permutation $kG$-modules with the following property: there exists a (possibly empty) sub-poset $\calY \subset s_p(G)$ which does not contain Sylow subgroups, is closed under conjugation and taking subgroups, and for which there exists some $i \in \Z$ such that for all $P \not\in \calY$, $\dim_k H_i(C(P)) = 1$ and $C(P)$ is exact in all other degrees. $C$ is homotopy equivalent to a complex $C'$ with the following property: there exists an $i \in \Z$ such that $C'_i = M \oplus N$, with $M$ a $\calY$-projective module or 0 and $N$ is an indecomposable $kG$-module with vertex set $\Syl_p(G)$, and for all $j \neq i$, either $C_j' = 0$ or $C_j'$ is a $\calY$-projective module.
	\end{lemma}
	\begin{proof}
        Note that if $\calY = \emptyset$, an argument similar to the proof of Theorem \ref{kernelthm} shows $C \simeq k_\omega[i]$ for some $i \in \Z$. Therefore, way may assume $\calY$ is nonempty.

		We may assume without loss of generality that $C$ contains no contractible summands. Let $S \in \Syl_p(G)$ and set $i = h(C(S))$. Then that each $(C(S))_i$ consists of direct sums of simple modules, as $k[N_G(S)/S]$ is semisimple. Therefore, we have a homotopy equivalence $k_{\omega}[i]\simeq C(S)$ for some $\omega \in \Hom(N_G(S)/S, k^\times)$. Therefore, $C_i$ contains a corresponding trivial source direct summand $E$, which satisfies $k_{\omega}[i]= E(S)$.

        We let $M$ be the summand of $C_{i-1}$ of minimal dimension containing $d_i(E)$, possibly $M = 0$ if $d_i(E)=0$. Now, define the two-term chain complex $D$ \[0 \to E \to M \to 0 \] with $E$ in degree $i$ and the nonzero differential induced by $d_i$ (possibly 0). We have a chain map $\phi$, with nonzero componentwise maps induced by inclusion.

		\begin{figure}[H]
			\centering
			\begin{tikzcd}
			\cdots \ar[r] & E \ar[d, hookrightarrow, "\phi_i"] \ar[r] & M \ar[d, hookrightarrow, "\phi_{i-1}"] \ar[r] & \cdots \\
			\cdots \ar[r] & C_i \ar[r, "d_i"] & C_{i-1} \ar[r] & \cdots
			\end{tikzcd}
		\end{figure}

		By construction, $\phi(S): D(S) \to C(S)$ is a homotopy equivalence. Let $P \not\in \calY$. We claim that $C(P)$ is homotopy equivalent to the chain complex $N[i]$, for some $kG$-module $N$ with $k$-dimension one. Indeed, for any $p$-subgroup $Q$ of $N_G(P)/P$, $C(P)(Q)$ satisfies $\dim_k H_i(C(P)(Q)) = 1$ and $C(P)(Q)$ is exact in all other degrees. Therefore, inductively deleting the acyclic direct summands of $C(P)$ yields a homotopy equivalence with $N[i]$, and since $\dim_k H_i(C(P)) = 1$, $N$ must have $k$-dimension one.

        We next claim that for any $p$-subgroup $P \not\in \calY$, $\phi(P): D(P) \to C(P)$ is a quasi-isomorphism. Suppose not, then either $(\phi_i (E))(P) \in \im d_{i+1}(P)$ or $(\phi_i(E))(P) \not\in \ker d_i$. However, in both cases, since $C(P) \simeq N[i]$, this would imply that $\phi(E)(P)$ lies in an indecomposable contractible summand of $C(P)$ of the form $0 \to N \to N \to 0$. An inductive argument up the poset $s_p(G)$ implies that $\phi(E)(S)$ also lies in an indecomposable contractible summand of $C(S)$ of the form $0 \to N' \to N' \to 0$, for some $k[N_G(S)/S]$-module $N'$ with $k$-dimension one as well, which contradicts our construction. Thus $\phi(P)$ is a quasi-isomorphism for all $P \not\in\calY$. Moreover, this shows that $d_i(P)$ is the zero map for all $P \not\in \calY$.

		Now, by the previous claim and Theorem \ref{boucthmgeneralization}, the mapping cone $C(\phi)$ is homotopy equivalent to a complex with $\calY$-projective components in all degrees. Therefore, by considering the structure of the mapping cone, $C$ is homotopy equivalent to an indecomposable complex with $\calY$-projective components in all degrees, with the exception of the summands $E$ of $C_i$ and $M$ of $C_{i-1}$. Replace $C$ with this reduced complex. If $M$ is $\calY$-projective, we are done. Otherwise, suppose $M$ has vertex $P \not\in \calY$. In this case, $C(P)$ contains exactly two nonzero components, $C_i = E(P)$ and $C_{i-1} = M(P)$. However, since $\phi(P)$ is a quasi-isomorphism, $(\im d_i(E))(P) = 0$, and since $C_{i-2}$ has no summands with vertex $P$, $(\ker d_{i-1}(M))(S) = M(S)$. Since these are the only two summands with vertex $P$, $H_{i-1}(C(P)) \cong M(P) \neq 0$, a contradiction. Thus $M$ has non-Sylow vertices, and we are done.
	\end{proof}

	As an application of this lemma, we obtain a structural result about faithful endotrivial complexes.

	\begin{corollary}
		Let $C \in \partial\calE_k(G)$, and let $\calX$ be the set of $p$-subgroups of $G$ which do not contain any nontrivial normal subgroups of $G$ as subgroups. Then there exists $i \in \Z$ for which $C_j$ is $\calX$-projective or $0$ for all $j \neq i$, and $C_i = M \oplus N$ where $M$ is $\calX$-projective or 0 and $N$ has Sylow vertices.
	\end{corollary}
    \begin{proof}
        This follows immediately from the previous lemma by setting $\calY$ to be the set of all $p$-subgroups of $G$ which contain no nontrivial normal subgroups of $G$ as subgroups.
    \end{proof}

	\section{Determining $\calE_k(G)$ for some groups} \label{section5}

	We can now completely deduce the structure of $\calE$ for some classes of well-understood groups. Most of the computations will rely on determining $\partial\calE_k(G)$.

	\begin{remark}
		Recall that a projective resolution of a module $M$ is a chain complex consisting of projective modules in each degree which is acyclic everywhere except in degree zero, where the homology is isomorphic to $M$. One source of endotrivial complexes comes from truncating periodic projective resolutions of the trivial $kG$-module $k$. We first will set some conventions. Assume that all projective resolutions are of the form $P \to M$ where \[P = \cdots \to P_1 \to P_0 \to 0\] is acyclic except in degree zero, $H_0(P)\cong M$ and $M \in {}_{kG}\mathbf{mod}$ in degree $-1$ of the chain complex $P \to M$ unless otherwise stated. By \textit{periodic resolution}, we mean a projective resolution $P$ for which there exist $i \in \N$ such that $\ker d_{i} \cong M$. These arise from \textit{periodic modules}, modules $M$ for which $M \cong \Omega^i(M)$ for some $i \in \N$.

		By the \textit{period} of a periodic resolution or module, we mean the minimum value of $i > 0$ satisfying $\Omega^i(M) = M$. We may \textit{truncate} a periodic resolution $P$ of period $i$ by taking the chain complex \[\hat{P} = 0 \to P_{i-1}\to \cdots \to P_0 \to M \to 0.\] It follows that $H_{i -1}(\hat{P}) \cong M$ and $H_{j}(\hat{P}) = 0$ for $j \neq i - 1$. Note that if the period of a projective resolution is $n$, the corresponding truncation has length $n+1$.

		It is well-known that the trivial $kG$-module $k$ is periodic if and only if $G$ has $p$-rank 1, in other words, $G$ has cyclic or quaternion Sylow $p$-subgroup. Let $S$ denote the Sylow $p$-subgroup of $G$. If $S = C_2$, $k$ has period 1, if $S = C_{p^n}$ for $p > 2$ or $p = 2$ and $n > 1$, then $k$ has period 2, and if $S = Q_{2^n}$, then $k$ has period 4. The first two periodic resolutions were given in Section \ref{faithfulsection} in the case of $p$-groups. We refer the reader to \cite[Chapter 12.7]{CE56} for an explicit construction of the periodic resolution of $k$ as $kQ_{2^n}$-module for $n \geq 3$.
	\end{remark}

	\begin{prop}\label{prk1gen}
		Let $G$ be a group with $p$-rank 1 and normal Sylow subgroup $S$, and let $C$ be a minimal truncation of the periodic resolution of the trivial module. Then $\langle C\rangle = \partial\calE_k(G)$.
	\end{prop}
	\begin{proof}
		In this case, $S$ is either cyclic or generalized quaternion, and in either case, the unique subgroup of $S$ of order 2 is normal in $G$. Therefore $\partial\calE_k(G)$ has rank at most 1, corresponding to the h-mark at the trivial subgroup. If $\partial\calE_k(G)$ was not generated by $C$, then Theorem \ref{relprojectivitylimitingthm} would imply the existence of a periodic projective resolution of $k$ with shorter period, contradicting minimality of the truncation.
	\end{proof}

	\subsection{$\calE_k(G)$ for abelian groups}

	We first deduce the structure of $\calE_k(G)$ for any abelian group $G$. In this situation, every subgroup is normal, which allows for restriction to preserve structural properties.

	\begin{prop}\label{abelianfaithfulmorphism}
		Let $H \leq G$ be Dedekind groups, i.e. groups for which every subgroup is normal. If $p \mid |G|$, then restriction induces an injective group homomorphism $\partial\calE_k(G) \to \partial \calE_k(H)$.
	\end{prop}
	\begin{proof}
		Any faithful chain complex $C$ must have h-marks zero and $H_0(C(P)) \cong k$ for all nontrivial $p$-subgroups $1 < P \in s_p(G)$. Since $p \mid |G|$, there exists at least one nontrivial $p$-subgroup, so by Theorem \ref{kernelthm}, the only faithful endotrivial complex with h-marks entirely zero is the trivial endotrivial complex $k$. So if $C$ is a nontrivial faithful endotrivial complex, $\Res^G_H C$ must have homology in a nonzero degree, and it follows that $\Res^G_H C$ also has h-mark 0 at $P\leq H$ and $H_0(C(P)) \cong k$ for all nontrivial $p$-subgroups $P \in s_p(H)$. Thus $\Res^G_H C$ is a trivial faithful endotrivial complex of $kH$-modules.
	\end{proof}

	\begin{prop}\label{prank2}
		Set $G = C_p \times C_p$. Then $\partial \calE_k(G)$ is trivial. In particular, if $p = 2$, \[\calE_k(G) = \langle  k[G/H_1] \to k, k[G/H_2] \to k, k[G/H_{3}] \to k, k[1] \rangle,\] and if $p$ is odd, \[\calE_k(G) = \langle k[G/H_1] \to k[G/H_1] \to k, \dots, k[G/H_{p+1}]\to k[G/H_{p+1}] \to k, k[1] \rangle,\] where $H_1,\dots, H_{p+1}$ are the $p+1$ subgroups of $G$ with index $p$, and the complexes are inflated truncated periodic free resolutions of $k$.
	\end{prop}
	\begin{proof}
		Suppose for contradiction that $\partial \calE_k(G) \neq 0$. Then there exists some nontrivial endotrivial chain complex $C$ for which $h(C(P)) = 0$ for all $1 < P \leq C_p \times C_p$, but $h(C) \neq 0$. Assume $h(C) > 0$, the other case follows similarly. By inductively deleting contractible summands of $C$, it suffices to assume $C_i = 0$ for $i < 0$. Therefore, by Lemma \ref{relprojectivitylimitingthm}, $C$ is homotopy equivalent to a truncated periodic free resolution of the $kG$-module $k$. However since $G$ has $p$-rank 2, no such resolution exists, a contradiction.

        Now, since there are $p+3$ subgroups of $G$, all of which are normal, the rank of $\calE_k(G)$ is at most $p+3$. However, if the rank was $p+3$, it would follow that there exists some endotrivial complex $C$ with h-mark at the trivial subgroup nonzero and h-mark at all other subgroups zero, contradicting that $\partial \calE_k(G) = 0$. Therefore $\calE_k(G)$ has rank at most $p+2$, and it follows from Theorem \ref{faithfuldecomp} that the generators in the proposition statement given form a basis of $\calE_k(G)$.
	\end{proof}

	From this, the structure of $\calE_k(G)$ for $G$ abelian follows.

	\begin{theorem}
		Let $G$ be an abelian group and let $s_p^1(G)$ denote the set of all (normal) $p$-subgroups of $G$ for which $G/P$ has $p$-rank at most 1. Then,
		\[\calE_k(G) \cong \prod_{P \in s_p^1(G)}\partial\calE_k(G/P).\] $\calE_k(G)$ is generated by endotrivial complexes which arise from inflating truncated periodic free resolutions of $k$ and shifts of $k$-dimension one representations.
	\end{theorem}
	\begin{proof}
		If $G/P$ has $p$-rank greater than 1, $\partial\calE_k(G/P) = 0$, since by Proposition \ref{abelianfaithfulmorphism}, we have an injective group homomorphism $\partial \calE_k(G/P) \to \partial\calE_k(H)$, where $H \leq G/P$ is an elementary abelian $p$-group of rank 2. However, $ \partial\calE_k(H) = 0$ by Proposition \ref{prank2}, so $\partial\calE_k(G/P) = 0$.

		Conversely, if $G/P$ has $p$-rank at most 1, this case was covered in Proposition \ref{prk1gen}. Additionally, if $G/P$ has $p$-rank 0, i.e. $P$ is the unique Sylow $p$-subgroup of $G$, then $\partial\calE_k(G/P) = \calE_k(G/P)$ has rank 1 as well. In this case $\partial\calE_k(G/P)$ is generated by $k[1]$ and has torsion subgroup $\Hom(G/P, k^\times)$.

		From the isomorphism $\partial\calE_k(G) \cong \prod_{P\in s_p^\triangleleft(G)} \partial \calE_k(G/P)$, the result follows.
	\end{proof}

	\begin{corollary}\label{thm:abeliansurjective}
		If $G$ is abelian, $\Lambda: \calE_k(G) \to O(T(kG))$ is surjective.
	\end{corollary}
	\begin{proof}
		\cite[Example 4.5]{M83} states that that if $G$ is abelian, \[B(G)^\times = \langle -[G/G], \{ [G/G] - [G/H] \mid [G:H] = 2\} \rangle.\] Under the decomposition afforded by $\kappa$ from \cite{BC23}, we have $O(T(kG)) = (B(S)^G)^\times \times \Hom(G,k^\times) \times \mathcal{L}_G$. However, since $N_G(P) = PC_G(P) = G$ for any $p$-subgroup $P$, $|\mathcal{L}_G| = 1$. Every element $\omega \in  \Hom(G,k^\times) \leq O(T(kG))$ satisfies that $(\kappa\circ \Lambda)( k_\omega[0]) = \omega$. Moreover, $k[1] \in \calE_k(G)$ satisfies $(\kappa\circ\Lambda)(k[1]) = -[S/S] \in (B(S)^G)^\times \leq O(T(kG))$, where $S \in \Syl_p(G).$ If $p$ is odd, we are done since in this case, $B(G)^\times = \{\pm [G/G]\}$.

		Otherwise, let $Q \leq S$ with $[S:Q] = 2$. Let $H$ be the unique largest $p'$-subgroup of $G$. Then, it is routine to verify that the chain complex $C_Q = k[G/QH] \to k$ is endotrivial, and that under restriction, $\Res^G_S C_Q = k[S/Q]\to k$, and that \[(\kappa\circ\Lambda )(C_Q)= [S/S] - [S/Q] \in (B(S)^G)^\times \leq O(T(kG)).\] Thus, $\Lambda$ is surjective.
	\end{proof}

	\subsection{$\calE_k(G)$ for $2$-groups with normal $2$-rank one}

	\begin{remark}\label{dihedralfaithfulendotrivialcomplex}
		We next focus on the $2$-groups with normal $2$-rank one (and additionally, the dihedral group $D_8$). The nonabelian $p$-groups with normal $p$-rank 1 are dihedral 2-groups, semidihedral 2-groups, and generalized quaternion 2-groups (see for instance \cite[Lemma 9.3.3]{Bou10}). Assume $k$ is a field of characteristic 2.

		We first focus on dihedral 2-groups, using the presentation \[D_{2^n} = \langle a ,b\mid a^{2^{n-1}} = b^2 = 1, {}^ba = a\inv\rangle.\] The normal subgroups of $D_{2^n}$ are as follows. There are 2 normal subgroups of index 2 isomorphic to $D_{2^{n-1}}$, which we denote $H^1_{2^n - 1} := \langle a^2, b \rangle$ and $H^2_{2^n - 1} := \langle a^2, ba \rangle.$ Next, for each $i \in \{0, \dots, n-1\}$, there is a copy of the cyclic group of $2^i$ elements of index $2^{n-i}$ which is normal in $D_{2^n}$, $H_{2^i} := \langle a^{2^{n - i + 1}} \rangle\cong C_{2^i}$. One may compute via the generators that for $k < n$, the quotient $D_{2^n}/H_{2^k}$ is isomorphic to $D_{2^{n-k}}$. Therefore, we have an isomorphism afforded by \ref{faithfuldecomp} \[\calE_k(D_{2^n})\cong \partial\calE_k(C_2)\times \partial\calE_k(C_2)\times \prod_{i=0}^n \partial\calE_k(D_{2^i}). \]

		We have already determined $\partial\calE_k(D_{2^i})$ for $i = 0,1,2$ (recalling $D_4 = V_4$ and $\partial\calE_k(V_4) = 0$), so it suffices to determine $\partial\calE_k(D_{2^i})$ for $i \geq 3$. In these cases, the computation is independent of $i \geq 3$. For $i \geq 3$, the only conjugacy classes of subgroups of $D_{2^i}$ not containing a nontrivial normal subgroup have representatives $\{1, \langle b\rangle, \langle ab \rangle\}.$ It follows that $\partial\calE_k(D_{2^i})$ consists of all endotrivial complexes for which $C(H) \simeq k[0]$ for all $H \leq D_{2^i}$ not conjugate to one of those three subgroups. Therefore, $\partial\calE_k(D_{2^i})$ has $\Z$-rank at most 3.

		Let $i \geq 3$. We first construct an endotrivial complex for $kD_{2^i}$ with $i \geq 3$ which is faithful. $D_{2^i}$ has three conjugacy classes of subgroups of order 2: $Z(D_{2^i}) = \langle a^{2^{i-2}} \rangle$, $\{\langle a^{2k}b\rangle \mid k \in \Z\}$, and $\{\langle a^{2k+1}b\rangle \mid k \in \Z\}$. Set $H_1 = \langle b\rangle$ and $H_2 = \langle ab \rangle$ as representatives of the conjugacy classes. We construct a chain complex of $kD_{2^i}$-modules, $\Gamma^D_i$, as follows:
		\[\Gamma^D_i := 0 \to kD_{2^i} \xrightarrow{d_2} k[D_{2^i}/H_1] \oplus k[D_{2^i}/H_2] \xrightarrow{d_{1}} k \to 0,\]
		\[d_2: x \in D_{2^i} \mapsto (xH_1, xH_2), \quad d_1: (xH_1, 0) \mapsto 1, (0, xH_2) \mapsto -1.\]

		Notice that $\Lambda(\Gamma^D_i) = [kD_{2^i}] + [k] - [k[D_{2^i}/H_1]] - [k[D_{2^i}/H_2]] \in \partial B(D_{2^i})^\times$, which generates $\partial B(D_{2^i})$ for $i \geq 3$, see \cite[Lemma 11.2.34]{Bou10}

	\end{remark}

	\begin{prop}
		$\Gamma_i^D$ is a faithful endotrivial complex of $kD_{2^i}$-modules.
	\end{prop}
	\begin{proof}
		$N_{D_{2^i}}(H_j) = Z(D_{2^i}) H_j$ for $j \in \{1,2\}$, and it follows that \[\Gamma^D_i(H_j) \cong (k[Z(D_{2^i}) H_j/H_j] \to k)\] is an endotrivial complex. Moreover, for every subgroup $H \leq D_{2^i}$ not $H_1, H_2,$ or $1$, $\Gamma^D_i(H) = k[0]$, the trivial endotrivial complex of $k[N_{D_{2^i}}(H)/H]$-modules. Therefore if $\Gamma^D_i$ is endotrivial, then it is faithful. It remains to show that $\Gamma^D_i$ has homology concentrated in exactly one degree, with that homology having $k$-dimension one.

        We show that $H_2(\Gamma^D_i)) \cong k$, $H_1(\Gamma^D_i) = 0$, and $H_{0}(\Gamma^D_i) = 0$. The final of these three assertions is clear since $d_1$ is surjective. By dimension counting, the first assertion holds if and only if the second does. It suffices to show $\ker d_1 \subseteq \im d_2$. Write, for $m \in k[G/H_1] \oplus k[G/H_2]$, \[m = \left(\sum_{h\in [G/H_1]} a_{h}h H_1, \sum_{h \in [G/H_2]} b_{h}hH_2\right).\] Therefore, \[m\in \ker d_0 \iff \sum_{h\in [G/H_1]} a_h = \sum_{h\in [G/H_2]} b_h, \quad a_h, b_h \in k.\]
		It follows that \[\ker d_1 = \text{span}_k \{(g_1H_1, g_2H_2): g_1,g_2 \in G\}.\] On the other hand, \[\im d_2 = \text{span}_k \{(gH_1, gH_2) : g\in G\}.\]
		So it suffices to show for any $g_1, g_2 \in G$, $(g_1H_1, g_2H_2) \in \text{span}_k\{(gH_1 ,gH_2): g\in G\} = \im d_2.$ First, observe
		\[(g_1H_1, g_1bH_2) = (g_1H_1, g_1H_1) - (g_1abH_1, g_1abH_2) + (g_1aH_1, g_1aH_1).\]

		Thus $(g_1H_1, g_1bH_2) = (g_1H_1, g_1babH_2) \in \im d_2$. It inductively follows that if $(g_1H_1, g_2H_2) \in \im d_2,$ then $(g_1H_1, g_2bH_2) = (g_1H_1, g_2babH_2) \in \im d_2$. Now, observe every $g \in G = \langle H_1H_2\rangle$ can be written either as $g = (bab)\cdots(bab)$ or $g = (bab)\cdots(bab)b$. In particular, this holds for $g_1\inv$ and $g_2$, so $(g_1H_1, g_1g_1\inv g_2H_2) = (g_1H_1, g_2H_2) \in \im d_2$.

        By Theorem \ref{endotrivdef2}, $\Gamma^D_i$ is endotrivial.
	\end{proof}

	In fact, $\Gamma^D_i$ generates $\partial\calE_k(D_{2^i})$.

	\begin{theorem}
		Let $i \geq 3$. Then $\partial\calE_k(D_{2^i}) = \langle \Gamma^D_i\rangle$.
	\end{theorem}
	\begin{proof}
		Suppose for contradiction that there exists a faithful endotrivial complex $C \in \partial\calE_k(D_{2^i})$ such that $[C] \not\in \langle\Gamma^D_i\rangle  \leq \partial \calE_k(D_{2^i})$. We consider the h-marks of $C$. Recall $h(\Gamma^D_i) = 2, h(\Gamma^D_i(\langle b\rangle)) = h(\Gamma^D_i(\langle ab \rangle)) = 1.$ Suppose $h(C) = d, h(C(\langle b \rangle)) = e, h(C\langle ab\rangle)) = f$.

		Since $C \not\in  \langle\Gamma^D_i\rangle$, at least one of $e,f$ must satisfy that $d \neq 2e$ or $d \neq 2f$. First, assume $d \neq 2e$. Set $C' = C\otimes_k (\Gamma^D_i)^{\otimes (-e)}$. Then $h(C') = d-2e \neq 0$, $h(C'(\langle b \rangle)) = e-e= 0$. We have $Z(D_{2^i}) = \langle a^{2^{i-2}} \rangle \cong C_2$. Now, $h(C'(\langle a^{2^{i-2}} \rangle)) = h(C'(\langle a^{2^{i-2}}, b \rangle)) = 0$ since $C' \in \partial\calE_k(D_{2^i})$, and finally since $\langle a^{2^{i-2}}b\rangle =_G \langle b \rangle$, $h(C'(\langle a^{2^{i-2}}b\rangle)) = 0$. Restricting $C'$ to $\langle a^{2^{i-2}},b \rangle \cong V_4$ yields an endotrivial complex $C'$ of $kV_4$-modules for which $h(C') \neq 0$ but $ h(C'(H)) = 0$ for all nontrivial subgroups $1 < H \leq V_4$. However, no such complex exists since $\partial \calE_k(V_4)$ is trivial. Thus, this case cannot occur.

		Otherwise, assume that $d = 2e$ but $d \neq 2f$. Set $C' = C\otimes_k (\Gamma^D_i)^{\otimes (-f)}$, then it follows by the same argument as before that $C'$ restricted to $\langle a^{2^{i-2}},ab \rangle \cong V_4$ yields a nontrivial endotrivial complex $C' \in \partial\calE_k(V_4)$, a contradiction.
	\end{proof}

	\begin{corollary}
		$\Lambda(-): \calE_k(D_{2^n}) \to O(T(kD_{2^n}))$ is surjective and $\rk_\Z\calE_k(D_{2^n}) = n+2$.
	\end{corollary}
	\begin{proof}
		$\Lambda: \partial\calE_k(D_{2^n}) \to \partial B(D_{2^n})^\times$ is surjective for $n \geq 3$, and $\Lambda: \calE_k(C_2) \to \partial B(C_2)^\times$ is surjective from Theorem \ref{thm:abeliansurjective}. Moreover, it follows from \cite[Example 4.5]{M83} that $\partial B(V_4)^\times = 0$. Thus, $\Lambda$ is surjective on all components of the decomposition \[\calE_k(D_{2^n})\cong \partial\calE_k(C_2)\times \partial\calE_k(C_2)\times \prod_{i=0}^n \partial\calE_k(D_{2^i}),\] hence surjective. The final statement is immediate.
	\end{proof}

	\begin{remark}
		We next determine $\calE_k(G)$ for generalized quaternion $2$-groups. For $Q_{2^n}$, $n \geq 3$, given by the presentation \[Q_{2^n} = \langle a,b \mid a^{2^{n-1}} = 1, {}^ba = a\inv, a^{2^{n-2}} = b^2\rangle,\] there are two normal subgroups of index 2 isomorphic to $Q_{2^{n-1}}$, which we label $H_{2^{n-1}}^1$ and $H_{2^{n-1}}^2$, and for each $i \in \{0, \dots, n-1\}$, there is a normal subgroup $H_{2^i} = \langle a^{2^{n-1-i}} \rangle\cong C_{2^i} $. We have an isomorphism $Q_{2^n}/H_{2^i}\cong D_{2^{n-i}} $.

		Therefore, \[\calE_k(Q_{2^n}) \cong \partial\calE_k(Q_{2^n}) \times \partial\calE_k(C_2)\times \partial\calE_k(C_2) \times \prod_{i=0}^{n - 1} \partial\calE_k(D_{2^i}).\] We determined in Proposition \ref{prk1gen} that $\partial\calE_k(Q_{2^n})$ is generated by a truncated periodic resolution of $k$ of period four. All other faithful constituents we have already determined, so we have a complete set of generators of $\calE_k(Q_{2^n})$.
	\end{remark}

	\begin{theorem}
		$\Lambda(-): \calE_k(Q_{2^n}) \to O(T(kQ_{2^n}))$ is surjective, and $\rk_\Z \calE_k(Q_{2^n}) = n+2$.
	\end{theorem}
	\begin{proof}
		We have shown previously that $\Lambda: \partial\calE_k(P)\to \partial B(P)^\times$ is surjective when $P$ is dihedral or cyclic. Moreover, $\partial B(Q_{2^n})$ is trivial, so $\Lambda$ is surjective on all components of the decomposition of $\calE_k(Q_{2^n})$, hence surjective. The final statement is immediate.
	\end{proof}

	\begin{remark}
		We now turn to semidihedral 2-groups. For $SD_{2^n}$, $n\geq 4$, given by the presentation \[\langle a,b \mid a^{2^{n-1}} = b^2 = 1, {}^ba = a^{2^{n-2}-1}\rangle.\] The normal subgroup structure is as follows; there are three subgroups of index two, $H_{2^{n-1}}^1 := \langle a^2, b\rangle \cong D_{2^{n-1}}$, $H_{2^{n-1}}^2 := \langle a^2, ab\rangle \cong Q_{2^{n-1}}$, and $H_{2^{n-1}} := \langle a \rangle \cong C_{2^{n-1}}$. For each $i \in \{0, \dots, n-2\}$ there is a normal subgroup $H_{2^i} := \langle a^{2^{n-1-i}} \rangle \cong C_{2^i}$. For $0 < i < n $, we have an isomorphism $SD_{2^n}/H_{2^i} \cong D_{2^{n-i}}$.

		Therefore, we have a decomposition \[\calE_{k}(SD_{2^n}) \cong \partial\calE_k(SD_{2^n})\times \partial\calE_k(C_2)\times \partial\calE_k(C_2) \times \prod_{i=0}^{n - 1} \partial\calE_k(D_{2^i}). \] It remains only to compute $\partial\calE_{k}(SD_{2^n})$.
    \end{remark}

    \begin{theorem}
        Let $n \geq 4$. $\rk_\Z \partial\calE_{k}(SD_{2^n}) = 1$.
    \end{theorem}
    \begin{proof}
        The group $SD_{2^n}$ has two conjugacy classes of subgroups of order 2, the center $Z := Z(SD_{2^n})$ and a full conjugacy class of subgroups contained in $H_{2^{n-1}}^1$. Let $H$ be a representative of this conjugacy class of noncentral subgroups, then $N_{SD_{2^n}}(H) = ZH \cong V_4$. Moreover, all subgroups of $SD_{2^n}$ of order at least 4 contain $Z(SD_{2^n})$, so the only h-marks which can be nonzero are at $1$ and $H$.

		Let $N = N_G(H)$. Suppose we have a faithful endotrivial complex $C \in \partial\calE_k(G)$. Restricting to $N \cong V_4$, we have that the h-marks at $Z$ and $N$ are 0 by faithfulness. Suppose the h-mark at $H$ is $i$, then the last subgroup of order 2 in $N$ is $SD_{2^n}$-conjugate to $H$, so it has h-mark $i$ as well. By the classification of $\calE_k(V_4)$, it follows that the h-mark at $1$ of $C$ is $h(C) = 2i$.

		Now, since $H_{2^{n-1}}^2\cap H = 1$, upon restriction to $H_{2^{n-1}}^2$, the h-mark of $C$ at 1 is $2i$ and 0 elsewhere. By the classification of $\partial\calE_k(Q_{2^n})$, $i$ must be even, that is, the h-mark at 1 must be a multiple of 4. We will construct a complex $\Gamma_n^S$ with $h(\Gamma_n^S) = 4$ and $h(\Gamma_n^S(H)) = 2$ and it follows follow that $\langle \Gamma_n^S \rangle_\Z = \partial\calE_k(SD_{2^n})$.

		To construct this complex, we first recall a theorem of Carlson and Th\'evenaz in their classification of $\calT(G)$ for $p$-groups.
    \begin{theorem}{\cite[Theorem 7.1]{CaTh00}}
        $\calT(G) \cong \Z \oplus \Z/2\Z$. The class of $\Omega(k)$ in $\calT(G)$ generates the torsion-free part of $\calT(G)$ and the class of $\Omega(\Delta(SD_{2^n}/H))$ in $\calT(G)$ is the lone nontrivial torsion element of $\calT(G)$.
    \end{theorem}

		Let $E = \Omega(\Delta(SD_{2^n}/H))$. We first define the chain complex corresponding to the construction of $E$, \[C := P \xrightarrow{c_2} k[SD_{2^n}/H] \xrightarrow{c_1} k,\] with $c_1:k[SD_{2^n}/H] \to k$ the augmentation homomorphism $\epsilon$, $P$ the projective cover of $\Delta(SD_{2^n}/H)$, and $c_2$ the corresponding covering composed with inclusion into $k[SD_{2^n}/H]$. Thus $H_2(C) = E$ and $H_i(C) = 0$ for $i \neq 2$. Moreover, it follows that $C(H) \cong k[N/H] \to k$, and $C(K) \cong k[0]$ for all $1, H\neq K \leq SD_{2^n}$.

        Set $D = C\otimes_k C$. It follows by the K\"unneth formula that $H_4(D) \cong E\otimes_k E$ and $H_i(D) = 0$ for $i \neq 4$. Moreover, by Proposition \ref{brauercommuteswithdualsandconj}, $H_4(D(H)) \cong k$, $H_i(D(H)) = 0$ for $i \neq 2$, and $D(K) \cong k[0]$ for $1,H \neq K \leq SD_{2^n}.$ Since $2[E] = 0 \in \calT(SD_{2^n})$, $E \otimes_k E \cong k \oplus P'$ for some projective $kG$-module $P'$. Denoting the differentials of $D$ by $\{d_i\}$, we have \[\ker d_4 \cong k \oplus P'\subset P^{\otimes 2} = D_4.\] Since $P'$ is an injective module as well, it follows that $P'$ is a direct summand of $D_4$, so we have a chain map $\pi_{P'}$ given by projection onto $P'$ as follows:

		\begin{figure}[H]
			\centering
			\begin{tikzcd}
			\cdots \ar[r] & D_4 \ar[r, "d_4"] \ar[d, "(\pi_{P'})_4"]& D_3 \ar[r] \ar[d, "0"]& \cdots \\
			\cdots \ar[r] & P' \ar[r] & 0 \ar[r] & \cdots
			\end{tikzcd}
		\end{figure}

		This is well-defined since $P'\subset \ker e_4$. Define $\Gamma_n^S := C(\pi_{P'})$, the mapping cone of $\pi_{P'}$. It follows that $H_4(\Gamma_n^S) \cong k$ and $H_i(\Gamma_n^S) = 0$ for all $i \neq 4$. Since we constructed $\Gamma_n^S$ by adding projective modules to a single component, the resulting complexes when taking the Brauer construction at any nontrivial subgroup remain unchanged, so Theorem \ref{endotrivdef2} implies $\Gamma_n^S$ is endotrivial, as desired.
    \end{proof}

	\begin{theorem}
		$\Lambda(-): \calE_k(SD_{2^n}) \to O(T(kSD_{2^n}))$ is surjective and $\rk_\Z \calE_k(SD_{2^n}) = n+2$.
	\end{theorem}
	\begin{proof}
		We have shown previously that $\Lambda: \partial\calE_k(P)\to \partial B(P)^\times$ is surjective when $P$ is dihedral, cyclic, or generalized quaternion. Moreover, $\partial B(SD_{2^n})$ is trivial, so $\Lambda$ is surjective on all components of the decomposition  \[\calE_{k}(SD_{2^n}) \cong \partial\calE_{kSD_{2^n}}\times \partial\calE_k(C_2)\times \partial\calE_k(C_2) \times \prod_{i=0}^{n - 1} \partial\calE_k(D_{2^i}), \], hence surjective. The final statement is immediate.
	\end{proof}

	\section{Not all orthogonal units lift to endotrivial complexes}\label{nonsurjectivesection}

	In this section, we describe a Galois invariance condition which elements in the image of $\Lambda: \calE_k(G) \to O(T(kG))$ satisfy. For this section, assume $k$ is a perfect field of characteristic $p$.

	\begin{definition}
		Let $\varphi: k\to k$ any field automorphism of $k$. $\varphi$ induces a ring automorphism on the group algebra $\varphi: kG\to kG$ which acts trivially on group elements. Given any $kG$-module $M$, precomposing by $\varphi\inv$ induces a new $kG$-module $^{\varphi}M:= \Iso_{\varphi\inv} M$. This induces an endofunctor $^\varphi(-): {}_{kG}\mathbf{mod} \to {}_{kG}\mathbf{mod}$.

		Denote the action of $kG$ on ${}^\varphi M$ by $\cdot_\varphi$. The action of $kG$ on  $^{\varphi}M$ is as follows: for $g \in G, m \in M, c \in k$:
		\[g\cdot_\varphi m = g m, \quad c\cdot_\varphi m = \varphi(c)m.\]

		Since $k$ is perfect of characteristic $p$, the Frobenius endomorphism \[F: k \to k,\quad x \mapsto x^p.\] is an automorphism of $k$.
	\end{definition}

	\begin{prop}\label{prop:propsoftwist}
        Let $k$ be a perfect field of characteristic $p > 0$.
		\begin{enumerate}
			\item $^\varphi(-)$ is an exact, additive functor which commutes with tensor products.
			\item $^\varphi(-)$ restricts to the identity functor on ${}_{kG}\textbf{perm}$ and restricts to an autoequivalence on ${}_{kG}\textbf{triv}$.
			\item For any $M \in {}_{kG}\mathbf{mod}$ and $P\in s_p(G)$, $^\varphi(M(P)) = ({}^\varphi M)(P)$.
			\item Let $\chi \in \Hom(G, k^\times)$ and let $k_\chi$ be the associated $k$-dimension one representation. Then ${}^\varphi k_\chi \cong k_{\varphi\inv \circ \chi}$.
		\end{enumerate}
	\end{prop}
	\begin{proof}
		(a) and (b) are straightforward. (c) follows from the property $^\varphi (M^P) = ({}^\varphi M)^P$ and since $^\varphi(-)$ does not alter the group action on $M$, the quotient term in the Brauer construction remains similarly unaltered. Since the functor does not alter morphisms, it is a natural isomorphism.

		For (d), regarding both ${}^\varphi k_\chi$ and $k_{\varphi\inv \circ \chi}$ as one-dimensional $k$-vector spaces, $\varphi\inv$ induces a map ${}^\varphi k_\chi \to k_{\varphi\inv \circ \chi}$. The map is bijective, and we claim it is a $kG$-module isomorphism. We compute, for $m \in {}^\varphi k_\chi$, $g \in G$, and $c \in k$:
		\[\varphi\inv (g\cdot_\varphi m) = \varphi\inv (gm) = \varphi\inv(\chi(g)m) = (\varphi\inv\circ \chi(g))\varphi\inv(m) = g\cdot \varphi\inv(m) \]
		\[\varphi\inv (c\cdot_\varphi m) = \varphi\inv (\varphi(c)m) = c\cdot\varphi\inv(m)\]
		Thus, ${}^\varphi k_\chi \cong k_{\varphi\inv \circ \chi}$.
	\end{proof}

	\begin{remark}
		$^\varphi(-)$ induces a ring automorphism on $T(kG)$ and $R_k(G)$. In the case of degree one Brauer characters, the image of $\varphi(\chi) \in R_k(G)$ is $\varphi\inv\circ\chi$, by the previous proposition. Because every orthogonal unit $u \in O(T(kG))$ can be expressed as a collection of virtual degree one characters via $\beta_G$ (see Remark \ref{trivsourceringdecomp}), and $^\varphi(-)$ commutes with the Brauer construction, \[\beta_G({}^\varphi u) = (\epsilon_P \cdot (\varphi\inv\circ\rho_P))_{P\in s_p(G)}.\] In other words, determining the image of $^\varphi u \in O(T(kG))$ amounts to post-composing $\varphi\inv$ to each local character. Similarly, for $\calE_k(G)$, \[\Xi({}^\varphi C) = (h_C(P), \varphi\inv \circ \calH_C(P))_{P \in s_p(G)}.\]
	\end{remark}

	We now focus on the Frobenius endomorphism, which is an automorphism since $k$ is assumed to be perfect. The fixed points of $F:k\to k$ is the subfield $\mathbb{F}_p$. Recall \[\calL_G = \left(\prod_{P\in s_p(G)} \Hom(N_G(P)/PC_G(P), k^\times)\right)'\leq O(T(kG)),\] the ``local homology'' subgroup of $O(T(kG))$.

	\begin{theorem}
		Let \[ u \in (B(S)^G)^\times \times \calL_G \leq O(T(kG)) \text{  with  } \beta_G(u) = (\epsilon_P \cdot \rho_P)_{P\in s_p(G)},\] and suppose $\rho_1$ is the trivial character on $G$. If there exists an endotrivial complex $C$ for which $\Lambda(C) = u$, then $^F(\rho_P) = \rho_P$ for all $P\in s_p(G)$.
	\end{theorem}
	\begin{proof}
		Let $C, u,$ and $(\epsilon_P \cdot \rho_P)_{P\in s_p(G)}$ be as above. Then $\calH_C(P) = \rho_P$, so ${}^FC \in \calE_k(G)$ as well. ${}^FC$ has the same h-marks as $C$, and has as local homology: \[\calH_{({}^FC)}(P) =  {}^F\calH_C(P) = {}^F(\rho_P),\] by Proposition \ref{prop:propsoftwist}. Now, $C\otimes_k ({}^FC)^*$ has h-marks entirely concentrated in degree 0. Since $\rho_1$ is the trivial representation and $\ker h = \{k_\omega[0] \mid \omega \in \Hom(G,k^\times)\}$, it follows that $\rho_P \cdot ({}^F\rho_P)^* $ is also the trivial representation. Thus $\rho_P = {}^F(\rho_P)$.
	\end{proof}

    \begin{definition}
        Say a function $f: G \to k$ is Frobenius-stable if $F\circ f = f,$ or equivalently, if $f$ descends to a function $f: G\to \mathbb{F}_p$. Say an orthogonal unit $u \in O(T(kG))$ is Frobenius-stable if each component in the $\calL_G$-constituent of $\kappa(u)$ is Frobenius-stable. Note that under this definition, for any $\omega \in \Hom(G,k^\times)$, $[k_\omega] \in O(T(kG))$ is Frobenius-stable, even if $\omega$ is not a Frobenius-stable group homomorphism.

    \end{definition}

	\begin{corollary}
		Let $u \in O(T(kG))$ with $\beta_G(u) = (\epsilon_P \cdot \rho_P)_{P \in s_p(G)}$. If there exists an endotrivial complex $C$ such that $u = \Lambda(C)$, then $u$ is Frobenius-stable.

		In particular, for ever $P \in s_p(G)$, $\rho_P \cdot \rho_1\inv|_{N_G(P)}$ admits values in $\mathbb{F}_p^\times$.
	\end{corollary}
	\begin{proof}
		This follows by applying a similar proof as in the previous theorem to the collection of signed tuples $\beta_G(u)$ (see Remark \ref{trivsourceringdecomp}) and twisting by $\rho_1$ so that the global homology is trivial. Note that the signs $\epsilon_P$ are $F$-invariant. The last statement follows from the previous definition.
	\end{proof}

	\begin{example} Let $p = 2$, $k$ be a finite field of characteristic 2 which has a 3rd root of unity $\omega$, and $G = A_4$. $kA_4$ has three projective indecomposables, and are given by $P_1 = k[A_4/C_3]$, $P_2 = k_{\omega}\otimes_k k[A_4/C_3]$, and $P_3 = k_{\omega^2}\otimes_k k[A_4/C_3]$, where $k_\omega$ is the simple representation of dimension one for which $(123)\cdot 1 = \omega$.

	Set $u = k_\omega + P_1 - P_2$. One may compute that $\beta_G(u) = (\chi_1, \chi_{C_2}, \chi_{V_4})$ with $\chi_{V_4} \in \Hom(A_4/V_4, k^\times)$ given by $(123)V_4 \mapsto \omega$, $\chi_{C_2} \in \Hom(V_4/V_4, k^\times)$ trivial, and $\chi_1 \in \Hom(A_4/A_4, k^\times)$ trivial. Therefore, $u \in O(T(kG))$. $^F(\chi_{V_4}) = k_{\omega^2}$, so $u$ cannot lift to an endotrivial complex.

	\end{example}

    \textbf{Acknowledgments:} The author is extremely grateful to his supervisor Robert Boltje for the countless hours of discussion, supervision, and assistance he has offered to made this paper possible. He additionally would like to thank the many mathematicians who offered their thoughts during the Dame Kathleen Ollerenshaw workshop including Nadia Mazza, Caroline Lassueur, and Markus Linckelmann, as well as the University of Manchester for their hospitality during the workshop. Finally, he would like to thank Dan Nakano for asking a very insightful question during a talk by the author, which led to the ideas present in Definition $\ref{hmarkhom}$.

	\bibliography{bib}
	\bibliographystyle{plain}

\end{document}